\documentclass{amsart}
\usepackage{amssymb}
\usepackage{amsfonts}
\usepackage{amsmath}
\usepackage[utf8]{inputenc}
\usepackage[bookmarks=true,colorlinks=true,linkcolor=blue,citecolor=blue]{hyperref}
\usepackage{graphicx}%
\usepackage{fancyhdr}
\usepackage{color}
\usepackage[mathlines]{lineno}
\usepackage{lscape}
\usepackage{epsfig}
\usepackage{natbib}
\usepackage{geometry}
\usepackage{tgbonum}
\fontfamily{qcr}\selectfont
\usepackage{float}
 \usepackage{tikz}

\theoremstyle{plain}

\newtheorem{definition}{Definition}

\newtheorem{proposition}{Proposition}

\numberwithin{equation}{section}

\begin{document}
\Large 
\title[Conditional entropies estimation : discrete case]{Conditional Shannon, R\'eyni, and Tsallis entropies estimation and asymptotic limits.}

\bigskip 

\noindent \author{Amadou Diadie {\sc Ba}, Gane Samb {\sc Lo}.}

\begin{abstract}
A method of estimating the joint probability mass function of a pair of discrete random variables is described. This estimator is used to construct the  conditional Shannon-R\'eyni-Tsallis entropies estimates. From there almost sure rates of convergence and asymptotic normality are 
established. The theorical results are validated by
simulations.

\end{abstract}
\maketitle

\bigskip \noindent 



 \section{Introduction}

 \subsection{Motivation}
  \noindent Let $X$ and $Y$ be two discrete random variables defined on a probability space $(\Omega, \mathcal{A}, \mathbb{P})$, with respectives values $x_1,\cdots,x_r$ and $y_1,\cdots,y_s$ (with $r> 1$ and $s>1$). \\
  
  \noindent The \textit{information amount} of (or \textit{contained} in) the outcome $(Y=y_j)$ given $(X=x_i)$ is (see \cite{carter}) $$\mathcal{I}(Y=y_j/X=x_i)=\log_2 \frac{p_{X,i}}{p_{i,j}}$$where $p_{X,i}=\mathbb{P}(X=x_i)$ and $ p_{i,j}=\mathbb{P}(X=x_i,Y=y_j)$.\\

\noindent The joint probability
distribution $\textbf{p}_{(X,Y)}=(p_{i,j})_{(i,j)\in I\times J}
$ of the events 
 $(X=x_i,Y=y_j)$, coupled with the \textit{information
amount} of every conditional event, $ \mathcal{I}(Y=y_j/X=x_i),$ forms a random variable whose
expected value is the \textit{conditional average amount of information}, or
\textit{conditional entropy} (more specifically, \textit{conditional Shannon entropy}), generated by this joint distribution.
\begin{definition}
Let  $X$ and $Y$ be two discrete random variables defined on a probability space $(\Omega, \mathcal{A}, \mathbb{P})$, taking respective values in the finite countable spaces 
$$X(\Omega)=\{x_1,x_2,\cdots,x_r\}\ \ \text{and}
\ \ Y(\Omega)=\{ y_1,\cdots,y_s\}\ \ (\text{with}\ \  r> 1 \ \ \text{and} \ \ s>1),$$ with respective probability distributions $\textbf{p}_X=(p_{X,i})_{(i\in I)}$, $ \textbf{p}_Y=(p_{Y,j})_{(j\in J)}$ where   $$ p_{X,i}=\mathbb{P}(X= x_i),\ \ i\in I=[1,r]\ \ \text{and}\ \   p_{Y,j}=\mathbb{P}(Y= y_j),\ \ j\in J=[1,s].$$

\noindent Let us denote by $\textbf{p}_{(X,Y)}=(p_{i,j})_{(i,j)\in I\times J}$, the probability distribution of the ordered pair $(X,Y)$ that is,  $$ p_{i,j}=\mathbb{P}(X=x_i,Y=y_j)\ \ \forall (i,j)\in I\times J.$$

\noindent (1)  The conditional Shannon entropy (\textit{CSE}) of $Y$, given $X$ is defined as (see \cite{cove})

\begin{eqnarray}
\label{cond-shan-ent}
 H(Y|X)=\mathbb{E}_{X,Y}\left[\log_2 \frac{p_{X}}{ p_{(X,Y)}}\right]&=&\sum_{(i,j)\in I\times J}p_{i,j}\,\log_2 \frac{p_{X,i}}{ p_{i,j}}.
\end{eqnarray}
 \end{definition}
 
 \noindent It is the average uncertainty of the variable $Y$ once $X$ is known.\\

 \noindent
  The entropy is usually measured in \textit{bit}s (\textbf{b}inary \textbf{i}nformation uni\textbf{t}) (if $\log_2$ is used), nats (if natural $\log$ is used), or hartley( if $\log_{10}$ is used), depending on the base of the logarithm which is used to define it.\\
 
 \noindent For ease of computations and notation convenience, we use the natural logarithm, since logarithms of varying bases are related by a constant.
\\

\noindent 
In what follows, 
 entropies will be considered as functions of p.m.f.'s, since they only take into account probabilities of specific events observed.\\

 \bigskip\noindent Additionally, the  joint Shannon entropy of $(X,Y)$ is defined as   \begin{equation*}\label{shan_ent_def}
  H(\textbf{p}_{(X,Y)})=\mathbb{E}_{X,Y}\left[\log \frac{1}{ p_{(X,Y)}}\right]=-\sum_{(i,j)\in I\times J }p_{i,j}\log p_{i,j}
\end{equation*}  and the Shannon entropy of the random variable $X$ is\begin{equation*}  H(\textbf{p}_X)=\mathbb{E}_{X}\left[\log \frac{1}{ p_{X}}\right]=-\sum_{i\in I}p_{X,i}\log p_{X,i}.
  \end{equation*} 

   \noindent See \textsc{Figure} \ref{rel} for depicting the relationships between information theoretic quantities.\\
   
 \noindent  $H(\textbf{p}_{(X,Y)})$ nat of information is needed on average to describe the exact state of the combined system determined by the two random variables $X$ and $Y$.  Now if we first learn the value of $X$, we have gained $H (\textbf{p}_{X})$ nats of information. Once $X$ is known, we only need $H(\textbf{p}_{(X,Y)})-H(\textbf{p}_{X})$ nats to describe the state of the whole system. This quantity is exactly $H(\textbf{p}_{(Y|X)})$, which gives the \textit{chain rule} of the CSE :
 \begin{equation*}
 H(\textbf{p}_{(Y|X)})= H(\textbf{p}_{(X,Y)})- H(\textbf{p}_X).
 \end{equation*}

 \noindent  Additionally the uncertainty of a random variable $X$ can never increase by knowledge of the outcome of another random variable this the \textit{monotonicity property} of CSE : $$H(\textbf{p}_{(X|Y)})\leq H(\textbf{p}_{X}).$$ 

   \vspace{2cm}
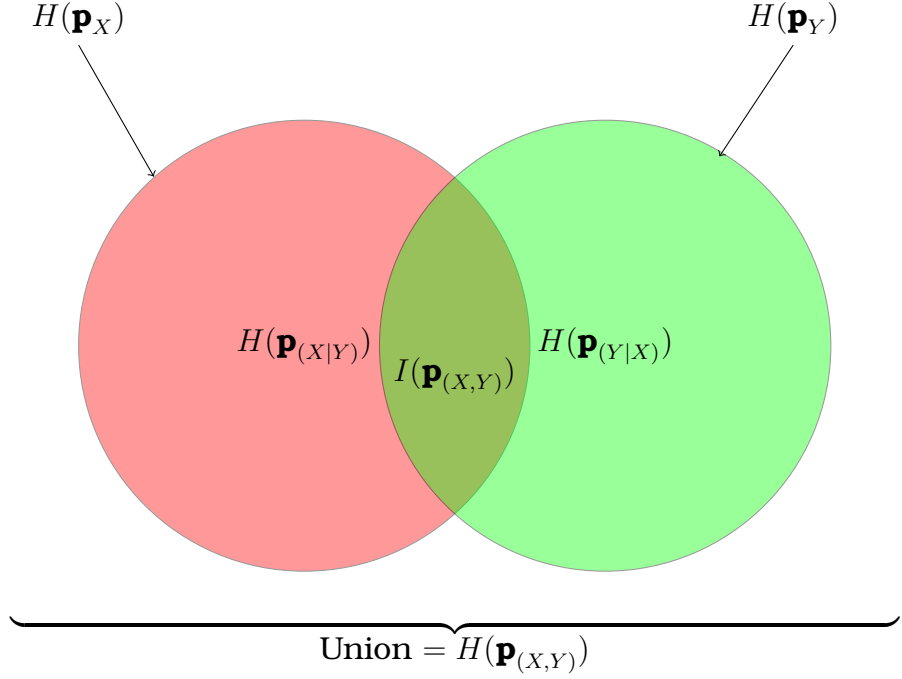
\begin{figure}

\begin{center}

   \begin{tikzpicture}
  \tikzset{venn circle/.style={draw,circle,minimum width=6cm,fill=#1,opacity=0.4,text opacity=1}}
  \node [venn circle = red] (A) at (0,0) {$H(\textbf{p}_{(X|Y)})$};
  \draw[<-] (-2,2.25) to (-3,4) 
   node[ above]{$H(\textbf{p}_X)$};
  \node [venn circle = green] (C) at (0:4cm) {$H(\textbf{p}_{(Y|X)})$};
  \draw[<-] (5.52,2.52) to (6.5,4) 
   node[ above]{$H(\textbf{p}_Y)$};
  \node[below] at (barycentric cs:A=1/2,C=1/2 ) {};   
  \node[below] at (barycentric cs:A=1,
  C=1 ) (endpoint) {$I(\textbf{p}_{(X,Y)})$};
\draw (2,-4)node[above]{$ \underbrace{\, \ \ \ \ \ \ \ \ \ \ \ \ \ \ \ \ \ \ \  \ \ \ \ \ \ \ \ \ \ \ \ \ \ \ \ \ \ \  \ \ \ \ \ \ \ \ \ \ \ \ \ \ \ \ \ \ \   \ \ \ \ \ \ \ \ \ \ \  \ \ \ \ \ \ \ \ \ \ \ \ \ \ \ \ \ \ \ }$};
\draw (2,-4.5)node[above]{\text{Union} $= H(\textbf{p}_{(X,Y)})$};
\end{tikzpicture}

\end{center}
  \caption{Venn Diagram depicting mutual information and entropy in a set-theory
way of thinking. The area contained by both circles is the joint entropy $H(\textbf{p}_{(X,Y)})$. The circle on the left (red and green) is the individual entropy $H(\textbf{p}_X)$, with the red being the conditional entropy $H(\textbf{p}_{(X|Y)})$. The circle on the right (green and red) is $H(\textbf{p}_Y)$, with the green being $H(\textbf{p}_{(Y|X)})$. The common area between $H(\textbf{p}_X)$ and $H(\textbf{p}_Y)$ at the middle is the mutual information $I(\textbf{p}_{(X,Y)})$.
}\label{rel}
 \end{figure}

\noindent  Conditional entropies play a central role in machine learning and applied statistics. There are many
problems where it is crucial for us to measure the uncertainty contained in a random variable if we observe an other random variable. (CSE) can be used to
capture these kind of uncertainty in information theory.  However it is insufficient in some other areas such as cryptography. Conditional R\'enyi entropy (CRE) and Conditional Tsallis entropy (CTE)) are more general, widely adopted in cryptography  as  a measure of security.  \\

\noindent Although this is a
fundamental problem in statistics and machine learning, interestingly, very little is known about how to
estimate these quantities efficiently in the discrete case.\\

 \noindent The goal of this
paper is to construct estimates of CSE and that of a family of CRE and TSE 
 and to establish their almost sure rates of convergence and asymptotic normality.\\

\noindent Unlike the \textit{CSE}, several definitions for \textit{CRE} have been proposed. 
 For example, \cite{arim} proposed a definition of CRE that found an application in information theory, \cite{jizba} proposed a definition of \textit{CRE} that found an application in time series analysis and
\cite{renn}, \cite{haya} and \cite{cach} proposed definitions of CRE that are suitable for cryptographic applications.
However, there is yet not a commonly accepted definition.\\

\noindent In this paper, we choose the only definition of \textit{CRE} for which the chain rule holds 
i.e.  \begin{equation*}
\label{chain_rule}R_\alpha(\textbf{p}_{Y|X})=R_\alpha(\textbf{p}_{X,Y})- R_\alpha(\textbf{p}_X),\ \
( \alpha>0,\ \ \alpha\neq 1),
\end{equation*}where $$R_\alpha(\textbf{p}_{X,Y})=\frac{1}{1-\alpha}\log \sum_{(i,j)\in I\times j}(p_{i,j})^\alpha$$ is the R\'enyi entropy of order $\alpha$ of $(X,Y)$ and 
$$R_\alpha(\textbf{p}_X)=\frac{1}{1-\alpha}\log \sum_{i\in I}(p_{X,i})^\alpha$$is the R\'enyi entropy of order $\alpha$ of $X$.

\bigskip \noindent (2)   The \textit{CRE} of order $\alpha$ of $Y$ given $X$ is defined as (see 
 \cite{jizba} and \cite{gols})\begin{equation}\label{cond-rey-def}
R_\alpha(\textbf{p}_{Y|X})=\frac{1}{1-\alpha}\log \left[ \frac{ \displaystyle \sum_{(i,j)\in I\times J}( p_{i,j})^\alpha}{\displaystyle \sum_{i\in I}(p_{X,i})^\alpha }
  \right],\ \
( \alpha>0,\ \ \alpha\neq 1).\end{equation}

\noindent Note that $R_\alpha(\textbf{p}_{Y|X})$ reduces to the $H(\textbf{p}_{Y|X})$ when $\alpha\rightarrow 1$ using the chain rule for $H(\textbf{p}_{Y|X})$ and for $R_\alpha(\textbf{p}_{Y|X})$.\\

\noindent For $\alpha=2$, $R_\alpha(\textbf{p}_{Y|X})$ refers as the \textit{conditional collision entropy},  a measure relevant for various hashing
schema and cryptographic protocols and for $\alpha\rightarrow+\infty$, it refers as the \textit{conditional min-entropy} (see \cite{gols}) traditionally used in Cryptography as a
measure of security.
\\

 \bigskip\noindent In \cite{gols}, some reasons are given for taking (\ref{cond-rey-def}) as the definition of the \textit{CRE}. This definition found an application in time series analysis (see \cite{jizba}).\\

\bigskip \noindent As a consequence, there is also no commonly accepted definition of \textit{CTE}.\\

\noindent In the following we seek a suitable definition for \textit{CTE} using a link between 
R\'enyi and Tsallis entropies.\\

\noindent We have 
\begin{eqnarray*}
 \label{talpha-rel} T_\alpha(\textbf{p}_{X})&=&\frac{1}{1-\alpha}\biggr[ \exp\left((1-\alpha)R_\alpha (\textbf{p}_{X})\right)-1\biggr]\\
\label{talpha2-rel} \text{and}\ \ T_\alpha(\textbf{p}_{(X,Y)})&=&\frac{1}{1-\alpha}\biggr[\exp\left((1-\alpha)R_\alpha (\textbf{p}_{(X,Y)})\right)-1\biggr],
\end{eqnarray*}
where 
\begin{eqnarray*}
T_\alpha(\textbf{p}_{X})=\frac{1}{1-\alpha}\left( \sum_{i\in I}(p_{X,i})^\alpha-1\right)
\end{eqnarray*}is the Tsallis entropy of order $\alpha$  of $X$ and \begin{eqnarray*}
T_\alpha(\textbf{p}_{(X,Y)})=\frac{1}{1-\alpha}\left( \sum_{(i,j)\in I\times J}(p_{i,j})^\alpha-1\right)
\end{eqnarray*}is the joint Tsallis entropy of order $\alpha$  of the pair $(X,Y)$.\\

\noindent Hence \textit{CTE} of random variable $Y$ given $X$ should satisfied 
\begin{equation*}\label{tsa_and_reyn}
T_\alpha(Y|X)=\frac{1}{1-\alpha}\biggr[\exp\left((1-\alpha)R_\alpha (Y|X)\right)-1\biggr],
\end{equation*}

\bigskip  \noindent This yields the following definition, using \eqref{cond-rey-def}.\\

\noindent (3) The \textit{CTE} of $Y$ given $X$ is given by (see \cite{abe} and \cite{mani})
\begin{eqnarray}\label{cond-tsalp-def}
T_\alpha(Y|X)=\frac{1}{1-\alpha}\left[\frac{ \displaystyle\sum_{(i,j)\in I\times J}( p_{i,j})^\alpha}{\displaystyle \sum_{i\in I}(p_{X,i})^\alpha}-1\right],\ \
( \alpha>0,\ \ \alpha\neq 1).
\end{eqnarray}

\noindent With this definition CTE satisfied the following pseudo-additivity property 
\begin{equation}\label{pseudo_addiv}
T_\alpha(\textbf{p}_{(X,Y)})=T_\alpha(\textbf{p}_{X})+T_\alpha(Y|X)+(1-\alpha)T_\alpha(\textbf{p}_{X})T_\alpha(Y|X).
\end{equation}

\noindent So we can conclude that relation \eqref{cond-tsalp-def} could better define \textit{CTE}
rather than the other definitions given in the litterature
. \\

\noindent When X and Y are independent random variables, we have
\begin{equation*}
T_\alpha(\textbf{p}_{(X,Y)})=T_\alpha(\textbf{p}_{X})+T_\alpha(Y)+(1-\alpha)T_\alpha(\textbf{p}_{X})T_\alpha(Y).
\end{equation*}

\bigskip \noindent From this small sample of entropies, we may give the following remarks: 

\bigskip \noindent For both  the CRE and CTE, we may have computation problems. So
without loss of generality, suppose that for any 
$\forall (i,j)\in I\times J$ \begin{eqnarray} \label{BD}
&&  p_{i,j}>0,\  \ \ p_{X,i}>0,\ \  \ p_{Y,j}>0. 
\end{eqnarray}
If Assumption (\ref{BD}) holds, we do not have to worry about summation problems, especially
for Tsallis, R\'enyi entropies, in the computations arising in estimation
theories. This explains why Assumption (\ref{BD}) is systematically used in a great number of
works in that topic, for example, in \cite{Singh2014Generalized}, \cite{Krishnamurthy2014Nonparametric}, \cite{Hall1987On}, and recently in \cite{DiadiBa2018Divergence} to cite a few.\\

\bigskip \noindent 
\noindent An important relation between CRE, CTE and 
 the \textit{joint power sums} (JPS) $\mathcal{S}_\alpha(\textbf{p}_{(X,Y)})$ and $\mathcal{S}_\alpha(\textbf{p}_{X})$ is 
\begin{eqnarray*}
R_\alpha(\textbf{p}_{(Y|X)})&=& \frac{1}{1-\alpha}\log\frac{\mathcal{S}_\alpha(\textbf{p}_{(X,Y)})}{\mathcal{S}_\alpha(\textbf{p}_{X})} ,\\ 
 \text{and} \  \  T_\alpha(\textbf{p}_{(Y|X)})&=&\frac{1}{1-\alpha}\left(\frac{\mathcal{S}_\alpha(\textbf{p}_{(X,Y)})}{\mathcal{S}_\alpha(\textbf{p}_{X})}-1\right).
\end{eqnarray*}
where
\begin{eqnarray*}\label{ialpha}
\mathcal{S}_\alpha(\textbf{p}_{(X,Y)})=\sum_{(i,j)\in I\times J}( p_{i,j})^\alpha\ \ \text{and}\ \  \mathcal{S}_\alpha(\textbf{p}_{X})=\sum_{i\in I}( p_{X,i})^\alpha.
\end{eqnarray*} 
\textit{Jayadev} \cite{jaya}  presents basic properties
of power sums of distributions.

\bigskip\noindent What is mostly considered so far for the R\'enyi and Tsallis 
entropies is its underlying axioms, but no specific asymptotic limits and central limit theorems from samples 
 have been given
for \textit{CSE}, \textit{CRE} and for \textit{CTE}
.
\\

\bigskip \noindent We propose in this paper  a plug-in approach that is essentially based on the estimation of the  joint probability distribution $\textbf{p}_{(X,Y)}
$ 
 from which, we can estimate the marginal distributions $\textbf{p}_X$, $\textbf{p}_Y$,
and then the quantities  $H(Y|X)$, $R_\alpha(\textbf{p}_{Y|X})$, and $T_\alpha(Y|X)$. This approach is motived by the fact that studying the joint distribution of  a pair of discrete random variables $X,Y$ taking values, respectively, in the finite sets $\mathcal{X}=\{x_i, i=1,\cdots,r\}$ and $\mathcal{Y}=\{y_j, j=1,\cdots,s\}$ is equivalent to studying the distribution of the $rs$ mutually exclusive possible values $(x_i,y_j)$ of   $(X,Y)$.\\

\noindent This allows us to transform the problem of estimating the joint discrete distribution of the pair $(X,Y)$ into the problem of estimating a simple distribution, say $p_Z$, of a single discrete random variable $Z$ suitably defined. Given an i.i.d sample of this latter random variable, we shall take,  as an estimator of the law $p_Z$, the associated empirical measure and plug it into formulas \eqref{cond-shan-ent},  \eqref{cond-rey-def}, and  \eqref{cond-tsalp-def} to obtain estimates for   $H(\textbf{p}_{Y|X})$, $R_\alpha(\textbf{p}_{Y|X})$, and $T_\alpha(Y|X)$.\\

\bigskip \noindent  
CSE, CRE, and CTE  have proven to be useful in
applications. 
\noindent CRE and CTE play important role in information theory, application in time series analysis and in cryptographic applications.\\

\noindent For instance, \textit{CRE} is applied in : cryptography (\cite{iwam}, \cite{cachin}), quantum systems \cite{vollb}), biomedical engineering (\cite{lake}, economics ( \cite{bent}),
fields related to statistics (\cite{kana}), etc.  \\

\noindent The concept of conditional entropies was proposed by  \cite{phil} and \cite{phil2} in investment in two securities
whose returns $X$ and $Y$ are two finite discrete random variables.
  Their theory has been proved useful by their empirical results. \cite{arack} used conditional entropy for network
traffic for anomaly detection.

 \subsection{ Overview of the paper} 
 The rest of the paper is organized as follows. In section \ref{maincontrib}, we define the auxiliary random variable $Z$ whose law is exactly the joint law of $(X,Y)$.
Section \ref{sect_estimation} is devoted to construct plug-in estimates of joint \textit{p.m.f.}'s of $(X,Y)$ and of CSE, CRE, and CTE
. Section \ref{main-res} establishes
consistency and asymptotic normality properties  of the estimates. 
Section \ref{sect_simul} provides a simulation study to assess the performence of our estimators and we finish by a conclusion in section \ref{sect_conclus}.

\section{Construction of the random variable $Z$ with law $\textbf{p}_{(X,Y)}$}
\label{maincontrib}

\noindent  Let $X$ and $Y$ two discrete random variables defined in the same probability space $(\Omega,\mathcal{A},\mathbb{P})$ and taking the following values $$x_1,x_2,\cdots,x_r\ \ \text{and}\ \ y_1,y_2,\cdots,y_s$$ respectively ($r>1$ and $s>1$).\\

  \noindent In addition let $Z$ a random variable defined on the same probability space $(\Omega,\mathcal{A},\mathbb{P})$ and taking the following values : $$z_1,z_2,z_3,z_4\cdots,z_{rs}.$$\\
 \noindent Denote $K=\{1,2,3,4\cdots,rs\}$.\\
 
 \noindent Simple computations give that for any $(i,j)\in I\times J$, we have $s(i-1)+j=\delta_i^j\in K $ and conversely for any $k\in K$ we have 
\begin{equation*}\label{convers}
\left(1+\lfloor\frac{k-1}{s}\rfloor,k-s\lfloor \frac{k-1}{s}\rfloor \right)\in I\times J,
\end{equation*} 
 where $ \lfloor x\rfloor$ denotes the largest integer less or equal to $x$.\\

\noindent 
For any possible joint values $(x_i,y_j)$ of the ordered pair $(X,Y)$, we assign the single value  
  $z_{\delta_i^j} $ of $Z$   
 such that 
 \begin{equation}\label{pijzk}
 \mathbb{P}(X=x_i,Y=y_j)=\mathbb{P}\left(Z=z_{\delta_i^j}\right),\ \ \text{where}\ \ \delta_i^j=s(i-1)+j,
 \end{equation}and conversely, for any possible value $z_k$ of $Z$, is assigned the single pair of values $\left( x_{1+\lfloor\frac{k-1}{s}\rfloor},y_{k-s\lfloor \frac{k-1}{s}\rfloor}
\right) $ such that   
 \begin{equation}\label{zkpij}
 \mathbb{P}(Z=z_k)=\mathbb{P}\left(X=x_{1+\lfloor\frac{k-1}{s}\rfloor},Y=y_{k-s\lfloor \frac{k-1}{s}\rfloor}\right).
 \end{equation}

\noindent This means that for any $(i,j)\in I\times J$, we have 
\begin{equation}
\label{pij3}
p_{i,j}= p_{Z,s(i-1)+j}
\end{equation} where $ p_{Z,k}=\mathbb{P}(Z=z_k)$ and conversely, for any $k\in K$
\begin{equation}\label{pzk2} 
p_{Z,k}=p_{1+\lfloor\frac{k-1}{s}\rfloor,k-s \lfloor\frac{k-1}{s}\rfloor}.
\end{equation}

\bigskip 
\noindent \textsc{Table} \ref{tabjpd} illustrates the correspondance between $p_{i,j}$ and $p_{Z,k}$, for ($i,j,k)\in I\times J\times K$.

\bigskip 
\noindent From there, the marginals \textit{p.m.f.'}s  $p_{X,i}$ and $p_{Y,j}$ and the conditionals \textit{p.m.f.'}s $p_{x_i|y_j}$ and $p_{y_j|x_i}$ are expressed from \textit{p.m.f.}'s of the random variable $Z$ by 
 \begin{eqnarray}
\label{pij2}&& p_{X,i}=\sum_{j=1}^sp_{Z,\delta_i^j},\ \
 \ p_{Y,j}=\sum_{i=1}^rp_{Z,\delta_i^j},\\
 && \notag p_{x_i|y_j}
 =\frac{p_{Z,\delta_i^j}}{\sum_{i=1}^rp_{Z,\delta_i^j}},\ \ \ \text{and}\ \ \ p_{y_j|x_i}
 =\frac{ p_{Z,\delta_i^j}}{\sum_{j=1}^sp_{Z,\delta_i^j}}
.
 \end{eqnarray}

\noindent Finally, CSE, CRE and CTE of $Y$ given $X$ are expressed simply in terms of $\textbf{p}_Z=(p_{Z,k})_{k\in K}$ through (\ref{pij3}), that is   
\begin{eqnarray*}
H(\textbf{p}_{Y|X})&=&-\sum_{(i,j)\in I\times J} p_{Z,\delta_i^j}\log \frac{p_{Z,\delta_i^j}}{p_{X,i}}, \\
 R_\alpha(\textbf{p}_{(Y|X)})&=& \frac{1}{1-\alpha}\log\left(\frac{\displaystyle  \sum_{(i,j)\in I\times J} (p_{Z,\delta_i^j})^\alpha}{\displaystyle \sum_{i\in I}(p_{X,i})^\alpha}\right)=\frac{1}{1-\alpha}\log\left(\frac{\displaystyle \sum_{k\in K}(p_{Z,k})^\alpha}{\displaystyle \sum_{i\in I}(p_{X,i})^\alpha}\right)
 ,\\ 
 \text{and} \  \  T_\alpha(\textbf{p}_{(Y|X)})&=&\frac{1}{1-\alpha}\left(\frac{\displaystyle  \sum_{(i,j)\in I\times J} (p_{Z,\delta_i^j})^\alpha}{\displaystyle \sum_{i\in I}(p_{X,i})^\alpha}-1\right)=\frac{1}{1-\alpha}\left(\frac{\displaystyle \sum_{k\in K}(p_{Z,k})^\alpha}{\displaystyle \sum_{i\in I}(p_{X,i})^\alpha}-1\right).
\end{eqnarray*}where $\displaystyle p_{X,i}$ and $p_{Y,j} $ are given by \eqref{pij2}.
\\

\noindent We have likewise \begin{eqnarray*}H(\textbf{p}_{X/Y})&=&-\sum_{(i,j)\in I\times J} p_{Z,\delta_i^j}\log \frac{p_{Z,\delta_i^j}}{p_{Y,j}},\\
R_\alpha(\textbf{p}_{(X/Y)})&=& \frac{1}{1-\alpha}\log\left(\frac{\displaystyle  \sum_{(i,j)\in I\times J} (p_{Z,\delta_i^j})^\alpha}{\displaystyle \sum_{j\in J}(p_{Y,j})^\alpha}\right)=\frac{1}{1-\alpha}\log\left(\frac{\displaystyle \sum_{k\in K}(p_{Z,k})^\alpha}{\displaystyle \sum_{j\in J}(p_{Y,j})^\alpha}\right),
\\
 \ \ \text{and} \  \  T_\alpha(\textbf{p}_{(X/Y)})&=&\frac{1}{1-\alpha}\left(\frac{\displaystyle \sum_{(i,j)\in I\times J} (p_{Z,\delta_i^j})^\alpha}{\displaystyle \sum_{j\in J}(p_{Y,j})^\alpha}-1\right)=\frac{1}{1-\alpha}\left(\frac{\displaystyle \sum_{k\in K}(p_{Z,k})^\alpha}{\displaystyle \sum_{j\in J}(p_{Y,j})^\alpha}-1\right).
\end{eqnarray*}

\begin{center}
\vspace{9ex}
\begin{table}
\centering
$$\begin{pmatrix}
p_{1,1}=p_{Z,1} & \cdots & p_{1,j}=p_{Z,j} & \cdots & p_{1,s}=p_{Z,s} \\ 
p_{2,1}=p_{Z,s+1} &\cdots & p_{2,j}=p_{Z,s+j} & \cdots &  p_{2,s}=p_{Z,2s} \\ 
\vdots & \vdots & \vdots & \vdots & \vdots \\ 
p_{i,1}=p_{Z,s(i-1)+1} & \cdots & p_{i,j}=p_{Z,\delta_i^j} & \cdots & p_{i,s}=p_{Z,si} \\ 
\vdots & \vdots & \vdots & \vdots & \vdots \\ 
p_{r,1}=p_{Z,s(r-1)+1} & \cdots & p_{r,j}=p_{Z,s(r-1)+j}&\cdots & p_{r,s}=p_{Z,rs}
\end{pmatrix} $$
$$\text{conversely}$$
$$\begin{array}{cccccc}
p_{Z,1}=p_{1,1} & p_{Z,2}=p_{1,2} & \cdots & p_{Z,k}=p_{1+\lfloor\frac{k-1}{s}\rfloor,k-s\lfloor \frac{k-1}{s}\rfloor
}  & \cdots & p_{Z,rs}=p_{r,s} \end{array} $$
\vspace{3ex} 
\caption{Illustration of the correspondance between $\textbf{p}_{(X,Y)}$ and $\textbf{p}_Z$.
}\label{tabjpd}
\end{table}
\end{center}

\section{Estimation}
 \label{sect_estimation}

\noindent In this section, we construct estimate of \textit{p.m.f.} $\textbf{p}_{Z,k}$ from i.i.d. random variables according to $\textbf{p}_Z$ and we give some inescapable results needed in the sequel, and finally construct the plug-in estimates of the entropies cited above. 
 \\
 
 \noindent Let $Z_1,\cdots,Z_n$ be $n$ i.i.d. random variables from $Z$ and according to $\textbf{p}_Z$. \\

\noindent Here, it is worth noting that, in the sequel, $K=\{1,2,\cdots,rs\}$, with $r$ and $s$ integers strictly greater than $1$. This means that 
    $rs$ can not be a prime number so that \eqref{pzk2} holds. \\

\noindent  For a given $k\in K$, define the easiest and most objective estimator of $p_{Z,k}$, based on the i.i.d sample $ Z_1,\cdots,Z_n,$ by 
 \begin{eqnarray}\label{pn}
 \widehat{p}_{Z,k}^{(n)} &=&\frac{1}{n}\sum_{\ell=1}^n1_{z_k}(Z_\ell)
\end{eqnarray}where 
 $1_{z_k}(Z_\ell)=\begin{cases}
 1\ \ \text{if}\ \ Z_\ell=z_k\\
 0\ \ \text{otherwise}.
 \end{cases} $\\

\bigskip \noindent This means that, for a given $(i,j)\in I\times J$, an estimate of $p_{i,j}$ based on the i.i.d sample $ Z_1,\cdots,Z_n,$ according to $\textbf{p}_Z$ is given by  
\begin{equation}\label{pxyni}
\widehat{p}_{i,j}^{(n)}= \widehat{p}_{Z,\delta_{i}^j}^{(n)}=\frac{1}{n}\sum_{\ell=1}^n1_{z_{\delta_i^j}}(Z_\ell).
\end{equation}
\noindent 
where $1_{z_{\delta_i^j}}(Z_\ell)=\begin{cases}
1\ \ \text{if}\ \ Z_\ell=z_{\delta_i^j}\\
 0\ \ \text{otherwise}.
\end{cases}$  \\

\bigskip \noindent Define the empirical probability distribution generated by i.i.d. random $Z_1,Z_2,\cdots,Z_n$ from the probability $\textbf{p}_Z$ as 
\begin{equation}\label{dpxyni}
\widehat{ \textbf{p}}_{(X,Y)}^{(n)}=(\widehat{p}_{i,j}^{(n)})_{(i,j)\in I\times J},
\end{equation}where $\widehat{p}_{i,j}^{(n)}$ is given by \eqref{pxyni}.\\

\bigskip \noindent From \eqref{pij2}, estimate of each of the marginals pdf's  $p_{X,i}$ and $ p_{Y,j}$ are 
\begin{eqnarray}\label{pxni}
\widehat{p}_{X,i}^{(n)}&=& \sum_{j=1}^s\widehat{p}_{Z,\delta_{i}^j}^{(n)}=\frac{1}{n}\sum_{\ell=1}^n1_{A_i}(Z_\ell)\\
\text{and} \notag\ \ \  &&\\
\widehat{p}_{Y,j}^{(n)}&=& \sum_{i=1}^r\widehat{p}_{Z,\delta_{i}^j}^{(n)}=\frac{1}{n}\sum_{\ell=1}^n1_{B_j}(Z_\ell),\label{pynj}
\end{eqnarray}
with  
\begin{eqnarray*}
A_i&=&\{z_{s(i-1)+1},z_{s(i-1)+2},\cdots,z_{si}\}=\bigcup_{j=1}^s \{z_{\delta_i^j}\}
\end{eqnarray*}
and \begin{eqnarray*}
 B_j&=&\{z_{j},z_{s+j},z_{2s+j},\cdots,z_{s(r-1)+j}\}=\bigcup_{i=1}^r\{ z_{\delta_i^j}\}.
\end{eqnarray*}
 
\bigskip \noindent In the sequel we use equally $p_{Z,k}$ or $p_{i,j}$  since they are equal in consideration of \eqref{pij3} and \eqref{pzk2}.\\ 

\noindent Before going further, let give some results concerning the empirical estimator $\widehat{p}_{Z,k}^{(n)}$ given by  \eqref{pn}. 
\\

\noindent For a given $k\in K$, this empirical estimator $\widehat{p}_{Z,k}^{(n)}$ is strongly consistent and asymptotically normal. Precisely, for a fixed $k\in K$, when $n$ tends to infinity,
 \begin{eqnarray*}
\label{pzn}
&&\widehat{p}_{Z,k}^{(n)}- p_{Z,k} \stackrel{a.s.}{\longrightarrow} 0,\\
&&\label{pnk}  \sqrt{n}(\widehat{p}_{Z,k}^{(n)}- p_{Z,k})  \stackrel{\mathcal{D}}{\rightsquigarrow}G_{p_{Z,k}}.
\end{eqnarray*}  where $G_{p_{Z,k}}\stackrel{d }{\sim}\mathcal{N}(0, p_{Z,k} (1-p_{Z,k} ))$.\\ 

\noindent These asymptotic properties derive from the law of large
numbers and central limit theorem. \\
 
\bigskip \noindent Here and in the following, $\stackrel{a.s.}{ \longrightarrow}$ means the \textit{almost sure convergence},  $\stackrel{\mathcal{D}}{ \rightsquigarrow}
$, the \textit{convergence in distribution}, and $\stackrel{d }{\sim}$, means \textit{equality in distribution}. \\

\noindent Recall that,
since for a fixed $k\in K,$ $n\widehat{p}_{Z,k}^{(n)}$ has a binomial distribution with parameters $n$  and success probability $p_{Z,k}$, we have 
 \begin{equation*}
 \mathbb{E}\left[ \widehat{p}_{Z,k}^{(n)}\right]=p_{Z,k} \ \ \text{and}\ \ \mathbb{V}\text{ar}(\widehat{p}_{Z,k}^{(n)})=\frac{p_{Z,k} (1-p_{Z,k} )}{n}.
\end{equation*}

\noindent Denote
\begin{eqnarray*}
\rho_n(p_{Z,k})=\sqrt{n/p_{Z,k}}\Delta_{p_{Z,k}}^{(n)}\ \ \text{and}\ \ a_{Z,n}=\sup_{k\in K}\left\vert  \Delta_{p_{Z,k}}^{(n)}
\right\vert,
\end{eqnarray*}where $\Delta_{p_{Z,k}}^{(n)}=\widehat{p}_{Z,k}^{(n)}- p_{Z,k}.$\\

\noindent By the asymptotic Gaussian limit of the multinomial law (see for example \cite{ips-wcia-ang}, Chapter 1, Section 4), we have
\begin{eqnarray*}
\label{pnj}&& \biggr( \rho_n(p_{Z,k}), \ k\in K\biggr)
\stackrel{\mathcal{D}}{\rightsquigarrow }G(\textbf{p}_Z),\ \ \ \ \text{as}\ \ n\rightarrow +\infty,
\end{eqnarray*}where $G(\textbf{p}_Z)= (G_{p_{Z,k}},k\in K)^t\stackrel{d }{\sim}\mathcal{N}(0,\Sigma_{\textbf{p}_Z}),$ and $\Sigma_{\textbf{p}_Z}$ is the covariance matrix which elements are :
\begin{eqnarray}\label{vars}
&&\sigma_{(k,k')}=(1-p_{Z,k} )1_{(k=k')}-\sqrt{ p_{Z,k}p_{Z,k'} } 1_{(k\neq k')}, \ \ (k,k') \in K^2.
\end{eqnarray}

 \bigskip \noindent 
By denoting $
a_{X,n}=\sup_{i\in I}|\widehat{p}_{X,i}^{(n)}-p_{X,i}|\ \ \text{and}\ \ a_{Y,n}=\sup_{j\in J}|\widehat{p}_{Y,j}^{(n)}-p_{Y,j}|
$ then, we have  
\begin{equation*}
\max(a_{X,n},a_{Y,n})\stackrel{a.s.}{\longrightarrow}0 \ \ \text{as}\ \ n\rightarrow +\infty. 
\end{equation*} 

\bigskip \noindent To finish, denote
\begin{eqnarray*}
\widehat{ \textbf{p}}_{(X,Y)}^{(n)}=(\widehat{p}_{i,j}^{(n)})_{(i,j)\in I\times J},\ \ \ 
 \ \  \widehat{\textbf{p}}_{X}^{(n)}=(\widehat{p}_{X,i}^{(n)})_{i\in I}\ \ \ \text{and}\ \ \widehat{\textbf{p}}_{Y}^{(n)}=(\widehat{p}_{Y,j}^{(n)})_{j\in J}.
\end{eqnarray*}

 \bigskip \noindent 
As a consequence, 
CSE, CRE 
and CTE are estimated from the sample $Z_1,\cdots,Z_n$ by their plug-in counterparts,
meaning that we simply insert the consistent \textit{p.m.f.} estimate $\widehat{p}_{Z,k}^{(n)}$ computed from \eqref{pn} in place of CSE, CRE,  
and CTE 
 expressions,  \textit{viz} 
\begin{eqnarray*}
&& H(\widehat{\textbf{p}}_{(Y|X)}^{(n)})=-\sum_{(i,j)\in I\times J}\widehat{p}_{Z,\delta_i^j}^{(n)}\log \frac{\widehat{p}_{Z,\delta_i^j}^{(n)}}{\widehat{p}_{X,i}^{(n)}},\ \ \  H(\widehat{\textbf{p}}_{(X/Y)}^{(n)})=-\sum_{(i,j)\in I\times J}\widehat{p}_{Z,\delta_i^j}^{(n)}\log \frac{\widehat{p}_{Z,\delta_i^j}^{(n)}}{\widehat{p}_{Y,j}^{(n)}}\\
 &&R_\alpha(\widehat{\textbf{p}}_{(Y|X)}^{(n)})=\frac{1}{1-\alpha}\log\left(\frac{\displaystyle\sum_{k\in K}\left(\widehat{p}_{Z,k}^{(n)}\right)^\alpha}{\displaystyle \sum_{i\in I}\left(\widehat{p}_{X,i}^{(n)}\right)^\alpha}\right),\ \ \  R_\alpha(\widehat{\textbf{p}}_{(X/Y)}^{(n)})=\frac{1}{1-\alpha}\log\left(\frac{\displaystyle\sum_{k\in K}\left(\widehat{p}_{Z,k}^{(n)}\right)^\alpha}{\displaystyle \sum_{j\in J}\left(\widehat{p}_{Y,j}^{(n)}\right)^\alpha}\right)\\
&&  T_\alpha(\widehat{\textbf{p}}_{(Y|X)}^{(n)})=\frac{1}{1-\alpha}\left(\frac{\displaystyle \sum_{k\in K}\left(\widehat{p}_{Z,k}^{(n)}\right)^\alpha}{\displaystyle \sum_{i\in I}\left(\widehat{p}_{X,i}^{(n)}\right)^\alpha}-1\right), \ \
 \text{and}\ \ T_\alpha(\widehat{\textbf{p}}_{(X/Y)}^{(n)})=\frac{1}{1-\alpha}\left(\frac{\displaystyle \sum_{k\in K}\left(\widehat{p}_{Z,k}^{(n)}\right)^\alpha}{\displaystyle \sum_{j\in J}\left(\widehat{p}_{Y,j}^{(n)}\right)^\alpha}-1\right)
\end{eqnarray*}
where $\widehat{p}_{Z,k}^{(n)}$, $\widehat{p}_{X,i}^{(n)},$ and $\widehat{p}_{Y,j}^{(n)},$ are given respectively by \eqref{pn}, and \eqref{pxni}-\eqref{pynj}.\\

\noindent In addition, define the joint power sum (\textit{JPS}) of the pair $(X,Y)$ estimate and the power sum (\textit{PS}) of $X$ estimate both  based on $Z_1,Z_2,\cdots,Z_n$ by 
\begin{eqnarray*}\label{jps_estima}
\mathcal{S}_\alpha(\widehat{\textbf{p}}_{(X,Y)}^{(n)})&=& \sum_{k\in K}\left(\widehat{p}_{Z,k}^{(n)}\right)^\alpha, \ \ \ \mathcal{S}_\alpha(\widehat{\textbf{p}}_{X}^{(n)})= \sum_{i\in I}\left(\widehat{p}_{X,i}^{(n)}\right)^\alpha.
\end{eqnarray*} and similarly for $Y$.

\section{Statements of the main results}\label{main-res}

\bigskip \noindent In this section, we state and prove almost sure consistency and central limit theorem  for the estimates defined above.

\bigskip \noindent (\textbf{A}) Asymptotic limits of CSE estimate $ H(\widehat{\textbf{p}}_{(Y|X)}^{(n)})$
.\\

\noindent Denote
 \begin{eqnarray*}
&&\label{asz} A_H(\textbf{p}_{(Y|X)})=\sum_{k\in K}\left\vert 1+ \log  ( p_{Z,k} )\right \vert\\
\label{sigz}\ \ \ 
&&\sigma^2(\textbf{p}_{(Y|X)})=\sum_{k\in K} p_{Z,k}(1- p_{Z,k} )(1+ \log  ( p_{Z,k} ))^2\\
&&\notag \ \ \ \ \ \ \ \  \ - \  2  \ \sum_{(k,k')\in K^2,k\neq k'}(p_{Z,k}p_{Z,k'})^{3/2}(1+\log (p_{Z,k}))(1+\log (p_{Z,k'})).\end{eqnarray*}
 \begin{proposition} \label{pro_cond_shan_yx} Let $\textbf{p}_{(X,Y)}$ a probability distribution and $\widehat{\textbf{p}}_{(X,Y)}^{(n)}$ be generated by i.i.d samples $Z_1,Z_2,\cdots,Z_n$ according to $\textbf{p}_{(X,Y)}$ and given by \eqref{dpxyni}, assumption \eqref{BD} be satisfied, the following results hold :
 \begin{eqnarray}\label{epas}
&&\limsup_{n\rightarrow +\infty}\frac{\left\vert H(\widehat{\textbf{p}}_{(Y|X)}^{(n)})-H(\textbf{p}_{(Y|X)})\right\vert}{ a_{Z,n} }\leq A_H(\textbf{p}_{(Y|X)}),\ \ \text{a.s.}\\
&&\sqrt{n}\left(H(\widehat{\textbf{p}}_{(Y|X)}^{(n)})-H(\textbf{p}_{(Y|X)}) \right)\stackrel{\mathcal{D}}{ \rightsquigarrow} \mathcal{N}(0,\sigma^2(\textbf{p}_{(Y|X)})),\ \ \text{as}\ \ n\rightarrow +\infty.\label{epnor}
\end{eqnarray}

 \end{proposition}
 \begin{proof}
 
 It is straightforward to write
 \begin{eqnarray*}
 H(\widehat{\textbf{p}}_{(Y|X)}^{(n)})-H(\textbf{p}_{Y|X})&=&-\sum_{(i,j)\in I\times J}\widehat{p}_{Z,\delta_i^j}^{(n)}\log \frac{\widehat{p}_{Z,\delta_i^j}^{(n)}}{\widehat{p}_{X,i}^{(n)}}+\sum_{(i,j)\in I\times J}p_{Z,\delta_i^j}\log \frac{p_{Z,\delta_i^j}}{p_{X,i}}\\
 &=&-\sum_{(i,j)\in I\times J}\widehat{p}_{Z,\delta_i^j}^{(n)}\log \widehat{p}_{Z,\delta_i^j}^{(n)}+
\sum_{(i,j)\in I\times J}\widehat{p}_{Z,\delta_i^j}^{(n)}\log\widehat{p}_{X,i}^{(n)}\\
 &&+\sum_{(i,j)\in I\times J}p_{Z,\delta_i^j}\log p_{Z,\delta_i^j}-\sum_{(i,j)\in I\times J}p_{Z,\delta_i^j}\log p_{X,i}\\
 &=&H(\widehat{\textbf{p}}_{(X,Y)}^{(n)})-H(\textbf{p}_{(X,Y)}) \\
 &&\ \ \ \ +
\sum_{(i,j)\in I\times J}[ \widehat{p}_{Z,\delta_i^j}^{(n)}\log\widehat{p}_{X,i}^{(n)}-p_{Z,\delta_i^j}\log p_{X,i}]
 \end{eqnarray*}where $$H(\textbf{p}_{(X,Y)})=-\sum_{(i,j)\in I\times J}p_{i,j}\log p_{i,j}=\sum_{k\in K}p_{Z,\delta_i^j}\log p_{Z,\delta_i^j}$$
 is the \textit{joint Shannon entropy} of the pair $(X,Y)$ and $$ H(\widehat{\textbf{p}}_{(X,Y)}^{(n)})=\sum_{(i,j)\in I\times J}\widehat{p}_{Z,\delta_i^j}^{(n)}\log \widehat{p}_{Z,\delta_i^j}^{(n)}$$ its estimate based on the i.i.d. sample $Z_1,Z_2,\cdots,Z_n$
 .\\
 
 \noindent For fixed $(i,j)\in I\times J,$ it holds that,  
 \begin{eqnarray*}
  \widehat{p}_{Z,\delta_i^j}^{(n)}\log\widehat{p}_{X,i}^{(n)}-p_{Z,\delta_i^j}\log p_{X,i}&=& 
  \widehat{p}_{Z,\delta_i^j}^{(n)}\log\frac{ \widehat{p}_{X,i}^{(n)} }{p_{X,i}} +(\widehat{p}_{Z,\delta_i^j}^{(n)}-p_{Z,\delta_i^j})\log p_{X,i}.
 \end{eqnarray*}
 
 \noindent 
 Therefore, we have, asymptotically \begin{eqnarray*}
 H(\widehat{\textbf{p}}_{(Y|X)}^{(n)})-H(\textbf{p}_{Y|X}) &\approx &H(\widehat{\textbf{p}}_{(X,Y)}^{(n)})
-H(\textbf{p}_{(X,Y)}) \end{eqnarray*} since $$\sup_{i\in I}| p_{X,i}^{(n)}-p_{X,i}|\stackrel{a.s.}{\longrightarrow}0 \ \ \text{and}\ \ \sup_{(i,j)\in I\times J}| \widehat{p}_{Z,\delta_i^j}^{(n)}-p_{Z,\delta_i^j}|\stackrel{a.s.}{\longrightarrow}0,$$  as, $n\rightarrow +\infty$.\\

 \noindent Finally \eqref{epas} and \eqref{epnor} follow from the Proposition 1 in \cite{ba-mutual-entrop}.

 \end{proof}

\bigskip
\noindent 
A similar proposition holds for the conditional entropy of $X$ given $Y$. 
\noindent The proof
is omitted being similar
as
 that of Proposition \ref{pro_cond_shan_yx}.

 \begin{proposition} \label{pro_cond_shan_xy} Under the same assumptions as in Proposition \ref{pro_cond_shan_yx}, the following results hold :
 \begin{eqnarray*}\label{epas_xy}
&&\limsup_{n\rightarrow +\infty}\frac{\left\vert H(\widehat{\textbf{p}}_{(X/Y)}^{(n)})-H(\textbf{p}_{(X/Y)})\right\vert}{ a_{Z,n} }\leq A_H(\textbf{p}_{(Y|X)}),\ \ \text{a.s.}\\
&&\sqrt{n}\left(H(\widehat{\textbf{p}}_{(X/Y)}^{(n)})-H(\textbf{p}_{(X/Y)}) \right)\stackrel{\mathcal{D}}{ \rightsquigarrow} \mathcal{N}(0,\sigma^2(\textbf{p}_{(Y|X)})),\ \ \text{as}\ \ n\rightarrow +\infty.\label{epnor_xy}
\end{eqnarray*}

 \end{proposition}

\bigskip \noindent (\textbf{B}) Asymptotic limit of CRE estimate $R_\alpha(\widehat{\textbf{p}}_{(Y|X)}^{(n)})$.\\

\noindent Denote
\begin{eqnarray}
\notag A_{R,\alpha}(\textbf{p}_{(Y|X)})
&=& \frac{\alpha}{\left\vert\alpha-1\right\vert   }\biggr[\frac{1}{\sum_{k\in K}\left(p_{Z,k}\right) ^{\alpha}} \sum_{k\in K}  (p_{Z,k}) ^{\alpha-1}\\
&&\notag \ \ \ \ \ \ \ \ \ \ \ + \ \ \  \frac{1}{\sum_{i\in I}( p_{X,i})^\alpha } \sum_{i\in I}  (p_{X,i}) ^{\alpha-1}\biggr],
\\
\label{sig_R_yx} \sigma_{R,\alpha}^2(\textbf{p}_{(Y|X)})
&=&\sigma_{\mathcal{R},\alpha}^{2}(\textbf{p}_X)+\sigma_{R,\alpha}^{2}(\textbf{p}_{(X,Y)})\\
\notag &&\ \ \ \ + \ \ \  2\, \text{Cov}
\left(G_{R,\alpha}(\textbf{p}_X),G_{R,\alpha}(\textbf{p}_{(X,Y)}) \right)
\end{eqnarray}
where 
\begin{eqnarray}
&&\notag G_{R,\alpha}(\textbf{p}_X) \stackrel{d}{\sim}\mathcal{N}(0,\sigma_{\mathcal{R},\alpha}^{2}(\textbf{p}_X))\ \
\text{with}\ \ \\ 
&&\label{sig_ralpha_x}\sigma_{\mathcal{R},\alpha}^{2}(\textbf{p}_X)=\left(\frac{\alpha}{(\alpha-1)\sum_{i\in I}( p_{X,i})^\alpha}\right)^2\biggr( \sum_{i\in I} (1- p_{X,i} ) (p_{X,i})^{2\alpha-1}\\
\noindent &&\nonumber\ \ \ \ \ \ \ \ \ \ \ \ \ \ \ \ \ \ \ \  \ \ \ \ \ \ \ \ \ \ \ \ \ \ \ \ \ \ \ \ - \ \ \ \ 2\sum_{(i,i')\in I^2,i\neq i'} (p_{X,i}p_{X,i'})^{\alpha-1/2}\biggr)
\end{eqnarray}and where 
\begin{eqnarray}
&&\label{gr_alpha_xy} G_{R,\alpha}(\textbf{p}_{(X,Y)}\stackrel{d}{\sim}\mathcal{N}(0,\sigma_{R,\alpha}^{2}(\textbf{p}_{(X,Y)})) \ \ \text{with}\\
&&\label{sig_ralpha_xy} \sigma_{R,\alpha}^{2}(\textbf{p}_{(X,Y)})=\left( \frac{\alpha
 }{(1-\alpha)\sum_{k\in K}\left(p_{Z,k}\right) ^{\alpha}}\right)^2\biggr[ \sum_{k\in K}( p_{Z,k})^{2\alpha-1}(1-p_{Z,k})\\
 &&\nonumber \ \ \ \ \ \ \ \ \ \ \ \ \ \ \ \ \ \ \ \  \ \ \ \ \ \ \ \ \ \ \ \ \ \ \ \ \ \ \ \ - \ \ \  2 \sum_{(k,k')\in K^2,k\neq k'}\left( p_{Z,k} p_{Z,k'}\right)^{\alpha-1/2}\biggr].
\end{eqnarray}

\begin{proposition}\label{pro_cond_reyn_yx}
 Under the same assumptions as in Proposition \ref{pro_cond_shan_yx}, the following asymptotic results hold
\begin{eqnarray}
&&\label{cond-rey-as} \limsup_{n\rightarrow+\infty}\frac{\left\vert R_\alpha(\widehat{\textbf{p}}_{(Y|X)}^{(n)})-R_\alpha(\textbf{p}_{(Y|X)})\right\vert}{a_{X,n}}\leq A_{R,\alpha}(\textbf{p}_{(Y|X)}),\ \ \text{a.s.}\\
&&\label{cond-rey-an} \sqrt{n}\left(R_\alpha(\widehat{\textbf{p}}_{(Y|X)}^{(n)})-R_\alpha(\textbf{p}_{(Y|X)}) \right)\stackrel{\mathcal{D} }{\rightsquigarrow}\mathcal{N}\left(0, \sigma_{R,\alpha}^2(\textbf{p}_{(Y|X)})\right),\ \ \text{as}\ \ n\rightarrow+\infty. 
\end{eqnarray}
\end{proposition}
\begin{proof}

\noindent For $\alpha\in (0,1)\cup (1,+\infty),$ we have 
\begin{eqnarray}\notag
  R_\alpha(\widehat{\textbf{p}}_{(Y|X)}^{(n)})-R_\alpha(\textbf{p}_{(Y|X)})&=&\frac{1}{1-\alpha}\biggr[\log  \frac{ \mathcal{S}_\alpha(\widehat{\textbf{p}}_{(X,Y)}^{(n)} ) }{S_\alpha(\widehat{\textbf{p}}_X^{(n)})}-\log \frac{ S_\alpha(\textbf{p}_{(X,Y)}) }{ S_\alpha(\textbf{p}_X)}\biggr] \\
\notag &=&\frac{1}{1-\alpha}\biggr[\log \mathcal{S}_\alpha(\widehat{\textbf{p}}_{(X,Y)}^{(n)} )- \log S_\alpha(\textbf{p}_{(X,Y)}) \\
\notag && \ \ \ \ \ \ - \ \ \ \ \left(\log S_\alpha(\widehat{\textbf{p}}_X^{(n)})-\log  S_\alpha(\textbf{p}_X)\right)\biggr]
 \\
 \label{ral_xy}&=&R_\alpha(\widehat{\textbf{p}}_{(X,Y)}^{(n)} ) - R_\alpha(\textbf{p}_{(X,Y)}) -\left(R_\alpha(\widehat{\textbf{p}}_X^{(n)})-R_\alpha(\textbf{p}_X)\right).
\end{eqnarray} Hence 
 \begin{eqnarray*}
 \left\vert R_\alpha(\widehat{\textbf{p}}_{(Y|X)}^{(n)})-R_\alpha(\textbf{p}_{(Y|X)})\right\vert &\leq &\left \vert R_\alpha(\widehat{\textbf{p}}_{(X,Y)}^{(n)} ) - R_\alpha(\textbf{p}_{(X,Y)})\right\vert \\
&&\ \ \ \  \ \ \ \ + \ \ \  \  \ \ \ \left\vert R_\alpha(\widehat{\textbf{p}}_X^{(n)})-R_\alpha(\textbf{p}_X)\right\vert.
 \end{eqnarray*}
\noindent From Proposition 3 in \cite{ba-mutual-entrop} and Corollary 3 in \cite{baentrop}, we have  respectively 
\begin{eqnarray*}&&\limsup_{n\rightarrow +\infty}\frac{\left \vert R_\alpha(\widehat{\textbf{p}}_{(X,Y)}^{(n)} ) - R_\alpha(\textbf{p}_{(X,Y)})\right\vert}{ a_{Z,n} }\leq A_{R,\alpha}(\textbf{p}_{(X,Y)}),\ \ \text{a.s}\\
\text{and}\ \ &&
 \limsup_{n\rightarrow+\infty}\frac{\left\vert R_\alpha(\widehat{\textbf{p}}_X^{(n)})-R_\alpha(\textbf{p}_X)\right\vert}{ a_{X,n} }\leq A_{R,\alpha}(\textbf{p}_X),\ \ \text{a.s.}
\end{eqnarray*}where

\begin{eqnarray*}
A_{R,\alpha}(\textbf{p}_{(X,Y)})&=&\frac{\alpha}{\left\vert1-\alpha\right\vert \mathcal{S}_\alpha(\textbf{p}_{(X,Y)})}\sum_{k\in K}\left(p_{Z,k}\right)^{\alpha-1},\\
\text{and}\ \ A_{R,\alpha}(\textbf{p}_X)&=& \frac{\alpha}{\left\vert\alpha-1\right\vert \mathcal{S}_\alpha(\textbf{p}_X) } \sum_{i\in I}  (p_{X,i}) ^{\alpha-1}.
\end{eqnarray*}

Thus
\begin{eqnarray*}
\limsup_{n\rightarrow+\infty}\frac{\left\vert  R_\alpha(\widehat{\textbf{p}}_{(Y|X)}^{(n)})-R_\alpha(\textbf{p}_{(Y|X)})\right\vert}{a_{X,n}}&\leq & A_{R,\alpha}(\textbf{p}_{(X,Y)})+A_{R,\alpha}(\textbf{p}_X),\ \ \ \text{a.s.}\\
&\leq &\frac{\alpha}{\left\vert\alpha-1\right\vert   }\biggr[\frac{1}{\mathcal{S}_\alpha(\textbf{p}_{(X,Y)})} \sum_{k\in K}  (p_{Z,k}) ^{\alpha-1}\\
&&\ \ \ \ \ +\ \ \ \ \frac{1}{ \mathcal{S}_\alpha(\textbf{p}_X) } \sum_{i\in I}  (p_{X,i}) ^{\alpha-1}\biggr],\ \ \ \text{a.s.} ,
\end{eqnarray*}as desired, and claimed by 
\eqref{cond-rey-as}. Let's prove the claim \eqref{cond-rey-an}. We have, from \eqref{ral_xy}, 
\begin{eqnarray*}
\sqrt{n}\left(  R_\alpha(\widehat{\textbf{p}}_{(Y|X)}^{(n)})-R_\alpha(\textbf{p}_{(Y|X)})\right) &=&\sqrt{n}( R_\alpha(\widehat{\textbf{p}}_{(X,Y)}^{(n)} )) -  R_\alpha(\textbf{p}_{(X,Y)}) \\
&&\ \ \ \ - \ \ \  \sqrt{n}(R_\alpha(\widehat{\textbf{p}}_X^{(n)})-R_\alpha(\textbf{p}_X)).
\end{eqnarray*}
\noindent So, since again from the Proposition 3 in \cite{ba-mutual-entrop} and Corollary 3 in
\cite{baentrop} 
\begin{eqnarray*}
&&\sqrt{n}\left( R_\alpha( \widehat{\textbf{p}}_{(X,Y)}^{(n)})-R_\alpha(\textbf{p}_{(X,Y)})\right)\stackrel{\mathcal{D}}{ \rightsquigarrow} 
\mathcal{N}\left( 0,\sigma _{R,\alpha}^{2}(\textbf{p}_{(X,Y)})\right)\text{ as } n\rightarrow + \infty\\
&&\sqrt{n}(R_\alpha(\widehat{\textbf{p}}_X^{(n)})-R_\alpha(\textbf{p}_X))\stackrel{\mathcal{D}}{ \rightsquigarrow} 
\mathcal{N}\left( 0,\sigma_{\mathcal{R},\alpha}^{2}(\textbf{p}_X)\right)\text{ as } n\rightarrow + \infty,
\end{eqnarray*}where $\sigma _{R,\alpha}^{2}(\textbf{p}_{(X,Y)})$  and $ \sigma_{\mathcal{R},\alpha}^{2}(\textbf{p}_X)$ are given by \eqref{sig_ralpha_xy} and \eqref{sig_ralpha_x} respectively. Therefore   
\begin{eqnarray*}
\sqrt{n}\left(R_\alpha(\widehat{\textbf{p}}_{(Y|X)}^{(n)})-R_\alpha(\textbf{p}_{(Y|X)}) \right)\stackrel{\mathcal{D} }{\rightsquigarrow}\mathcal{N}\left(0,\sigma_{R,\alpha}^2(\textbf{p}_{(Y|X)})  \right),\ \ \text{as}\ \ n\rightarrow+\infty
\end{eqnarray*}where $\sigma_{R,\alpha}^2(\textbf{p}_{(Y|X)})$ is given by \eqref{sig_R_yx}. This proves the claim \eqref{cond-rey-an} and ends the proof of the proposition \ref{pro_cond_reyn_yx}.
 \end{proof}
 
\bigskip \noindent \noindent 
A similar proposition holds for the conditional R\'enyi entropy of $X$ given $Y$.
 
\noindent The proof
is omitted being similar
as that of Proposition \ref{pro_cond_reyn_yx}.\\

\noindent Denote

\begin{eqnarray*}
\notag A_{R,\alpha}(\textbf{p}_{(X/Y)})
&=& \frac{\alpha}{\left\vert\alpha-1\right\vert   }\biggr[\frac{1}{\sum_{k\in K}\left(p_{Z,k}\right) ^{\alpha}} \sum_{k\in K}  (p_{Z,k}) ^{\alpha-1}\\
&&\notag \ \ \ \ \ \ \ \ \ \ \ + \ \ \  \frac{1}{ \sum_{j\in J}  (p_{Y,j}) ^{\alpha}  } \sum_{j\in J}  (p_{Y,j}) ^{\alpha-1} \biggr],
\\
\label{sig_R_xy} \sigma_{R,\alpha}^2(\textbf{p}_{(X/Y)})&=&\sigma _{\mathcal{R},\alpha}^{2}(\textbf{p}_{Y})+\sigma _{R,\alpha}^{2}(\textbf{p}_{(X,Y)})\\
\notag && \ \ \ \  + \ \ \ 2\, \text{Cov}
\left(G_{R,\alpha}(\textbf{p}_{Y}),G_{R,\alpha}(\textbf{p}_{(X,Y)})\right),
\end{eqnarray*}
where 
\begin{eqnarray*}
&& G_{R,\alpha}(\textbf{p}_{Y}) \stackrel{d}{\sim}\mathcal{N}(0,\sigma _{\mathcal{R},\alpha}^{2}(\textbf{p}_{Y}))
\ \ \text{with}\\
&&\sigma _{\mathcal{R},\alpha}^{2}(\textbf{p}_{Y})=\frac{\alpha^2}{(1-\alpha)\left(\sum_{j\in J}( p_{Y,j})^\alpha
\right)^2}\biggr[\sum_{j\in J} (1- p_{Y,j} ) (p_{Y,j})^{2\alpha-1}\\
&&\nonumber \ \ \ \ \  \ \ \ \ \ \ \ \ \ \ \  \ \ \  \ \ \ \ \ \ \ \  \ \ \ - \ \ \  2\sum_{(j,j')\in J^2,j\neq j'}(p_{Y,j}p_{Y,j^{'}})^{\alpha-1/2}\biggr].
\end{eqnarray*} and $G_{R,\alpha}(\textbf{p}_{(X,Y)})$ is given by \eqref{gr_alpha_xy}.

\begin{proposition}\label{pro_cond_reyn_xy}
Under the same assumptions as in Proposition \ref{pro_cond_shan_yx}, the following asymptotic results hold
\begin{eqnarray}
&&\label{cond-rey-as_xy} \limsup_{n\rightarrow+\infty}\frac{\left\vert R_\alpha(\widehat{\textbf{p}}_{(X/Y)}^{(n)})-R_\alpha(\textbf{p}_{(X/Y)})\right\vert}{a_{Y,n}}\leq A_{R,\alpha}(\textbf{p}_{(X/Y)}),\ \ \text{a.s.}\\
&&\label{cond-rey-an_xy} \sqrt{n}\left(R_\alpha(\widehat{\textbf{p}}_{(X/Y)}^{(n)})-R_\alpha(\textbf{p}_{(X/Y)}) \right)\stackrel{\mathcal{D} }{\rightsquigarrow}\mathcal{N}\left(0, \sigma_{R,\alpha}^2(\textbf{p}_{(X/Y)})\right),\ \ \text{as}\ \ n\rightarrow+\infty. 
\end{eqnarray}
\end{proposition}

 \bigskip \noindent (\textbf{E})  
Asymptotic limit of CTE estimate $T_\alpha(\widehat{\textbf{p}}_{(Y|X)}^{(n)})$
.\\

\noindent Denote

\begin{eqnarray}
\notag A_{T,\alpha}(\textbf{p}_{(Y|X)})&=&\frac{\alpha}{ |1-\alpha| \sum_{i\in I}( p_{X,i})^\alpha
} \biggr[ \frac{ \sum_{k\in K}\left(p_{Z,k}\right) ^{\alpha}}{\sum_{i\in I}( p_{X,i})^\alpha
} \sum_{i\in I}( p_{X,i})^{\alpha-1}\\
&&\notag \ \ \ \ \ \ \  + \ \ \ \sum_{k\in K}\left(p_{Z,k}\right)^{\alpha-1}\biggr]\\
\label{sig_T_yx}\sigma_{T,\alpha}^2(\textbf{p}_{(Y|X)} )&=&\sigma_{T,\alpha}^2(\textbf{p}_X)+\sigma_{T,\alpha}^2(\textbf{p}_{(X,Y)}) \\
\notag &&\ \ \ \ + \ \ \ 2\, \text{Cov}
\left(G_{T,\alpha}(\textbf{p}_X),G_{T,\alpha}(\textbf{p}_{(X,Y)}) \right)\end{eqnarray}where 
\begin{eqnarray*}
&&\notag \label{gt_alpha_x} G_{T,\alpha}(\textbf{p}_X) \stackrel{d}{\sim}\mathcal{N}(0,\sigma_{T,\alpha}^2(\textbf{p}_X))\ \ \text{with}\\
&& \label{sig_talpha_x} \sigma_{T,\alpha}^2(\textbf{p}_X)=
\left(\frac{\alpha}{1-\alpha}\right)^2\left(\frac{ \sum_{k\in K}\left(p_{Z,k}\right) ^{\alpha}   }{\left(\sum_{i\in I}( p_{X,i})^\alpha
\right)^2} \right)^2  \biggr[\sum_{i\in I} (1- p_{X,i} ) (p_{X,i})^{2\alpha-1}\\
&&\nonumber \ \ \ \  \ \ \ \ \ \ \  \ \ \ \ \ \ \ \ \ \ \ \ \  \ \ \  \ \ \ \ \ \ \  \ \ \ - \ \ \  2\sum_{(i,i')\in I^2,i\neq i'}(p_{X,i}p_{X,i^{'}})^{\alpha-1/2}\biggr]
\end{eqnarray*}and where
\begin{eqnarray}
&& \label{gt_alpha_xy} G_{T,\alpha}(\textbf{p}_{(X,Y)}) \stackrel{d}{\sim}\mathcal{N}(0,\sigma_{T,\alpha}^2(\textbf{p}_{(X,Y)}))\ \ \text{with}\\
&&  \label{sig_talpha_xy} \sigma_{T,\alpha}^2(\textbf{p}_{(X,Y)})=\left( \frac{\alpha}{1-\alpha}\right)^2\left( \frac{ 1}{\left(\sum_{i\in I}( p_{X,i})^\alpha
\right)^2}\right)^2\biggr[ \sum_{k\in K}(1-p_{Z,k})( p_{Z,k})^{2\alpha-1}\\
&&\nonumber \ \ \ \  \ \ \ \ \ \ \  \ \ \ \ \ \ \ \ \ \ \ \ \  \ \ \  \ \ \ \ \ - \ \ \ \  2 \sum_{(k,k')\in K^2,k\neq k'}\left( p_{Z,k} p_{Z,k'}\right)^{\alpha-1/2}\biggr].
\end{eqnarray}

\begin{proposition}\label{pro_cond_tsal_yx}
Under the same assumptions as in Proposition \ref{pro_cond_shan_yx}, the following asymptotic results hold \begin{eqnarray}\label{tsal_cond_as}
&&\limsup_{n\rightarrow+\infty}\frac{ |T_\alpha(\widehat{\textbf{p}}_{(Y|X)}^{(n)})-T_\alpha\left(\textbf{p}_{(Y|X)}\right)|}{a_{Z,n}}\leq   A_{T_\alpha}\left(\textbf{p}_{(Y|X)}\right)\ \ ,\text{a.s.}\\
&&\sqrt{n}\left(T_\alpha(\widehat{\textbf{p}}_{(Y|X)}^{(n)})-T_\alpha(\textbf{p}_{(Y|X)})\right)\stackrel{\mathcal{D} }{\rightsquigarrow}\mathcal{N}\left(0,\sigma_{T,\alpha}^2(\textbf{p}_{(Y|X)}\right),\ \ \text{as}\ \ n\rightarrow+\infty.\label{tsal_norm}
\end{eqnarray}
\end{proposition}
\begin{proof}For $\alpha\in (0,1)\cup(1,+\infty)$, we have 
\begin{eqnarray*}
T_\alpha(\widehat{\textbf{p}}_{(Y|X)}^{(n)})- T_\alpha(\textbf{p}_{(Y|X)})&=&\frac{1}{1-\alpha}\left( \frac{\displaystyle \sum_{k\in K}\left(\widehat{p}_{Z,k}^{(n)}\right)^\alpha}{\displaystyle \sum_{i\in I}\left(\widehat{p}_{X,i}^{(n)}\right)^\alpha}-\frac{\displaystyle \sum_{k\in K}(p_{Z,k})^\alpha}{\displaystyle \sum_{i\in I}(p_{X,i})^\alpha}\right)\\
&=&\frac{1}{1-\alpha}\left(\frac{ \mathcal{S}_\alpha(\widehat{\textbf{p}}_{(X,Y)}^{(n)} ) }{S_\alpha(\widehat{\textbf{p}}_X^{(n)})}-\frac{ S_\alpha(\textbf{p}_{(X,Y)}) }{ S_\alpha(\textbf{p}_X)}\right).
\end{eqnarray*}
\noindent $T_\alpha(\widehat{\textbf{p}}_{(Y|X)}^{(n)})- T_\alpha(\textbf{p}_{(Y|X)})$ can be re-expressed as 
\begin{eqnarray*}
T_\alpha(\widehat{\textbf{p}}_{(Y|X)}^{(n)})- T_\alpha(\textbf{p}_{(Y|X)})&=&\frac{1}{1-\alpha}\biggr[\mathcal{S}_\alpha(\widehat{\textbf{p}}_{(X,Y)}^{(n)} )\left(\frac{ S_\alpha(\textbf{p}_X)-S_\alpha(\widehat{\textbf{p}}_X^{(n)}) }{S_\alpha(\widehat{\textbf{p}}_X^{(n)})S_\alpha(\textbf{p}_X)}\right)\\
&&\ \ \ +\ \ \frac{1}{S_\alpha(\textbf{p}_X)} ( \mathcal{S}_\alpha(\widehat{\textbf{p}}_{(X,Y)}^{(n)} )- S_\alpha(\textbf{p}_{(X,Y)}))\biggr].
\end{eqnarray*}
\noindent Asymptotically, using Proposition 2 in \cite{ba-mutual-entrop} , we have 
\begin{eqnarray*}
&&\label{asymprep}  T_\alpha(\widehat{\textbf{p}}_{(Y|X)}^{(n)})- T_\alpha(\textbf{p}_{(Y|X)}) \approx \frac{1}{1-\alpha}\biggr[\frac{  S_\alpha(\textbf{p}_{(X,Y)}) }{(S_\alpha(\textbf{p}_X))^2}(S_\alpha(\textbf{p}_X)-S_\alpha(\widehat{\textbf{p}}_X^{(n)}))\\
 &&\ \ \ \ \ \ \ \ \ \ \ \ \ \ \ \ \ \ \ \ \ \ \ \ \ \ \ \ \ \ \ \ \ \ \ \ \ \ \  +\ \ \ \frac{1}{S_\alpha(\textbf{p}_X)} ( \mathcal{S}_\alpha(\widehat{\textbf{p}}_{(X,Y)}^{(n)} )- S_\alpha(\textbf{p}_{(X,Y)})) \biggr]
\end{eqnarray*}
 \noindent First from Corollary 1 in
\cite{baentrop} and Proposition 2 in \cite{ba-mutual-entrop}, we have respectively 
\begin{eqnarray*}
&& \limsup_{n\rightarrow+\infty}\frac{ \left\vert S_\alpha(\widehat{\textbf{p}}_X^{(n)})-S_\alpha(\textbf{p}_X)\right \vert }{a_{X,n}} \leq  A_{\mathcal{S}_\alpha}(\textbf{p}_{X}),\ \ \ \text{a.s.}\\
\text{and}\ \ 
&& \limsup_{n\rightarrow+\infty}\frac{\left \vert \mathcal{S}_\alpha(\widehat{\textbf{p}}_{(X,Y)}^{(n)} )- S_\alpha(\textbf{p}_{(X,Y)})\right\vert
}{a_{Z,n}}\leq  A_{\mathcal{S}_\alpha}(\textbf{p}_{(X,Y)}) ,\ \ \ \text{a.s.}
\end{eqnarray*}
\noindent where 
\begin{eqnarray*}
A_{\mathcal{S}_\alpha}(\textbf{p}_{X})= \alpha \sum_{i\in I}( p_{X,i})^{\alpha-1}\ \
\text{and}\ \  A_{\mathcal{S}_\alpha}(\textbf{p}_{(X,Y)})&=& \alpha\sum_{k\in K}\left(p_{Z,k}\right)^{\alpha-1},
\end{eqnarray*}

\noindent which entails
\begin{eqnarray*}
\limsup_{n\rightarrow+\infty}\frac{\left\vert  T_\alpha(\widehat{\textbf{p}}_{(Y|X)}^{(n)})- T_\alpha(\textbf{p}_{(Y|X)})
\right\vert}{a_{Z,n}} 
&\leq &\frac{\alpha}{|1-\alpha|}\biggr[ \frac{\displaystyle \sum_{k\in K}\left(p_{Z,k}\right) ^{\alpha}}{\displaystyle\left( \sum_{i\in I}( p_{X,i})^\alpha\right)^2
}\sum_{i\in I}( p_{X,i})^{\alpha-1}\\
&& \ \ \ \ \ \ + \ \ \ \ \ \ \frac{1}{ \sum_{i\in I}( p_{X,i})^\alpha} \sum_{k\in K}\left(p_{Z,k}\right)^{\alpha-1}\biggr]
\\
&\leq &\frac{\alpha}{ |1-\alpha| \sum_{i\in I}( p_{X,i})^\alpha
} \biggr[ \frac{ \sum_{k\in K}\left(p_{Z,k}\right) ^{\alpha}}{\sum_{i\in I}( p_{X,i})^\alpha
} \sum_{i\in I}( p_{X,i})^{\alpha-1}\\
&&\ \ \ \ \ \ \ \ \ \ \ \ + \ \ \   \sum_{k\in K}\left(p_{Z,k}\right)^{\alpha-1}\biggr],\ \ \ \text{a.s.}
\end{eqnarray*}which proves the claim \eqref{tsal_cond_as}.\\

\noindent  Second from Corollary 1 in
\cite{baentrop} and Proposition 2 in \cite{ba-mutual-entrop}, we have, respectively, as $ n\rightarrow+\infty,$
\begin{eqnarray*}
&&\sqrt{n}(S_\alpha(\textbf{p}_X)-S_\alpha(\widehat{\textbf{p}}_X^{(n)})\stackrel{\mathcal{D} }{\rightsquigarrow}\mathcal{N}(0,\sigma_{\mathcal{S}_\alpha}^2(\textbf{p}_X))
\\
\text{and}\ \  &&\sqrt{n}\left(\mathcal{S}_\alpha(\widehat{\textbf{p}}_{(X,Y)}^{(n)} )-\mathcal{S}_\alpha(\textbf{p}_{(X,Y)})\right)\stackrel{\mathcal{D} }{\rightsquigarrow}\mathcal{N}(0,\sigma_{\mathcal{S}_\alpha}^2(\textbf{p}_{(X,Y)}))
\end{eqnarray*}

where \begin{eqnarray*}
\sigma_{\mathcal{S}_\alpha}^2(\textbf{p}_X)&=&\alpha^2\biggr[ \sum_{i\in I} (1- p_{X,i} ) (p_{X,i})^{2\alpha-1}\\
&&\ \ \ \ -\ \   2\sum_{(i,i')\in I^2,i\neq i'}(p_{X,i}p_{X,i^{'}})^{\alpha-1/2}\biggr]\\
\text{and}\ \ 
\sigma_{\mathcal{S}_\alpha}^2(\textbf{p}_{(X,Y)})&=&\alpha^2\biggr[ \sum_{k\in K}(1-p_{Z,k})( p_{Z,k})^{2\alpha-1}\ \ \ \\
&&\ \ \ \ - \ \  2 \sum_{(k,k')\in K^2,k\neq k'}\left( p_{Z,k} p_{Z,k'}\right)^{\alpha-1/2}\biggr]
.
\end{eqnarray*}

\noindent So that, as $n\rightarrow+\infty$, 
\begin{eqnarray*}
&& \sqrt{n}\left( T_\alpha(\widehat{\textbf{p}}_{(Y|X)}^{(n)})- T_\alpha(\textbf{p}_{(Y|X)})\right) \stackrel{\mathcal{D} }{\rightsquigarrow} \frac{1}{1-\alpha}\biggr[\frac{  S_\alpha(\textbf{p}_{(X,Y)}) }{(S_\alpha(\textbf{p}_X))^2}G_{\mathcal{S}_\alpha}(\textbf{p}_X)\\
\notag &&\ \ \ \ \ \ \ \ \ \ \ \ \ \ \ \ \ \ \ \ \ \ \ \ \ \ \ \ \ \ \ \ \ \ \ \ \ \ \  +\ \ \ \frac{1}{S_\alpha(\textbf{p}_X)} G_{\mathcal{S}_\alpha}(\textbf{p}_{(X,Y)}) \biggr]
\end{eqnarray*} where \begin{eqnarray*}
G_{\mathcal{S}_\alpha}(\textbf{p}_X)\stackrel{d }{\sim} \mathcal{N}(0,\sigma_{\mathcal{S}_\alpha}^2(\textbf{p}_X))\ \ \text{and}\ \ G_{\mathcal{S}_\alpha}(\textbf{p}_{(X,Y)})\stackrel{d }{\sim}  \mathcal{N}(0,\sigma_{\mathcal{S}_\alpha}^2(\textbf{p}_{(X,Y)})).
\end{eqnarray*}
\noindent Therefore  \begin{eqnarray*}
\sqrt{n}\left(T_\alpha(\widehat{\textbf{p}}_{(Y|X)}^{(n)})-T_\alpha\left(\textbf{p}_{(Y|X)}\right)\right)\stackrel{\mathcal{D} }{\rightsquigarrow}\mathcal{N}\left(0,\sigma_{T,\alpha}^2(\textbf{p}_{(Y|X)} )\right),\ \ \text{as}\ \ n\rightarrow+\infty.
\end{eqnarray*}where $\sigma_{T,\alpha}^2(\textbf{p}_{(Y|X)} )$ is given by \eqref{sig_T_yx}.\\

\bigskip \noindent This proves the claim \eqref{tsal_norm} and ends the proof of the Proposition \ref{pro_cond_tsal_yx}.
\end{proof}

\bigskip \noindent \noindent 
A similar proposition holds for the conditional Tsallis entropy of $X$ given $Y$.
\\

\noindent 
The proof
is omitted being similar as that of Proposition \ref{pro_cond_tsal_yx}.\\

\noindent Denote
\begin{eqnarray*}
\notag A_{T,\alpha}(\textbf{p}_{(X/Y)})&=&\frac{\alpha}{ |1-\alpha| \sum_{j\in J}( p_{Y,j})^\alpha
} \biggr[ \frac{ \sum_{k\in K}\left(p_{Z,k}\right) ^{\alpha}}{ \sum_{j\in J}( p_{Y,j})^\alpha
}  \sum_{j\in J}( p_{Y,j})^{\alpha-1}\\
&&\notag \ \ \ \ \ \ \  + \ \ \ \sum_{k\in K}\left(p_{Z,k}\right)^{\alpha-1}\biggr]\\
\label{sig_T_xy}\sigma_{T,\alpha}^2(\textbf{p}_{(X/Y)}&=&\sigma_{T,\alpha}^2(\textbf{p}_Y) + \sigma_{T,\alpha}^2(\textbf{p}_{(X,Y)}) +2\, \text{Cov}
\left(G_{T,\alpha} (\textbf{p}_Y),G_{T,\alpha}(\textbf{p}_{(X,Y)})  \right)\end{eqnarray*}where 
$G_{T,\alpha}(\textbf{p}_{(X,Y)})$ and $\sigma_{T,\alpha}^2(\textbf{p}_{(X,Y)})$ are as in \eqref{gt_alpha_xy} and \eqref{sig_talpha_xy}
and 
\begin{eqnarray*}\label{gt_alpha_y}
&&G_{T,\alpha} (\textbf{p}_Y)\stackrel{d }{\sim}\mathcal{N}\left(0,\sigma_{T,\alpha}^2(\textbf{p}_Y)\right)\ \ \text{with}\\
&&\sigma_{T,\alpha}^2(\textbf{p}_Y)=\left( \frac{\alpha}{1-\alpha}\right)^2\left( \frac{ \sum_{k\in K}\left(p_{Z,k}\right) ^{\alpha}}{\left(\sum_{j\in J}( p_{Y,j})^\alpha
\right)^2}\right)^2\biggr[\sum_{j\in J}( 1-p_{Y,j}) (p_{Y,j})^{2\alpha-1}\\
&&\nonumber \ \ \ \ \  \ \ \  \ \ \ \ \ \ \ \  \ \ \ - \ \ \  2\sum_{(j,j')\in J^2,j\neq j'}(p_{Y,j}p_{Y,j^{'}})^{\alpha-1/2}\biggr].
\end{eqnarray*}

\begin{proposition}\label{pro_cond_tsal_xy}
Under the same assumptions as in Proposition \ref{pro_cond_shan_yx}, the following asymptotic results hold \begin{eqnarray*}\label{tsal_cond_as_xy}
&&\limsup_{n\rightarrow+\infty}\frac{ \left\vert T_\alpha(\widehat{\textbf{p}}_{(X/Y)}^{(n)})-T_\alpha(\textbf{p}_{(X/Y)})\right\vert}{a_{Z,n}}\leq   A_{T_\alpha}\left(\textbf{p}_{(X/Y)}\right)\ \ ,\text{a.s.}\\
&&\sqrt{n}\left(T_\alpha(\widehat{\textbf{p}}_{(X/Y)}^{(n)})-T_\alpha(\textbf{p}_{(X/Y)})\right)\stackrel{\mathcal{D} }{\rightsquigarrow}\mathcal{N}\left(0,\sigma_{T,\alpha}^2(\textbf{p}_{(X/Y)}\right),\ \ \text{as}\ \ n\rightarrow+\infty.\label{tsal_norm_xy}
\end{eqnarray*}
\end{proposition}
\section{Simulation study} \label{sect_simul}

\noindent In this section, we present a example to confirm the consistency and the asymptotic normality of the proposed measures of information estimators developed in the previous sections. \\

 \noindent For simplicity consider two discretes random variables $X$ and $Y$ having each one two outcomes $x_1,x_2,x_3$ and $y_1,y_2$ and such that 

\begin{eqnarray*}
&&\mathbb{P}(X=x_1,Y=y_1)=	\frac{3600}{5369},\ \ \ \mathbb{P}(X=x_1,Y=y_2)=\frac{900}{5369}\\
&&\mathbb{P}(X=x_2,Y=y_1)=\frac{400}{5369},\ \  \mathbb{P}(X=x_2,Y=y_2)=\frac{225}{5369}\\
&&\mathbb{P}(X=x_3,Y=y_1)=\frac{144}{5369},\ \  \mathbb{P}(X=x_3,Y=y_2)=\frac{100}{5369}.
\end{eqnarray*}

\noindent So that the associated random variable $Z$, defined by \eqref{pijzk} and \eqref{zkpij}, is a discrete random variable whose probability distribution is that of a \textit{discrete Zipf distributions} $Z_{\beta,m}$ with parameter $\beta=2$ and $m=6
$. 
Its \textit{p.m.f.}  is defined by 
$$p_{Z,k}=\frac{k^{-\beta}}{\displaystyle \sum_{i=1}^{m}i^{-\beta}}\ \ \text{for}\ \ k =1,2,\cdots,m$$
where $\sum_{j=1}^mj^{-\beta}$ refers to the generalized harmonic function.\\
\noindent From \eqref{cond-shan-ent}, \eqref{cond-rey-def} and \eqref{cond-tsalp-def}, we obtain \begin{eqnarray*}
 H(\textbf{p}_{Y|X})
 &=&0.52623\ \ \text{nat},\ \ \ R_2(\textbf{p}_{Y|X})= 0.39027\ \ \text{nat},\ \ \text{and}\ \ T_2(\textbf{p}_{Y|X})=0.32312\ \ \text{nat}.
\end{eqnarray*}
\noindent As well as 
\begin{eqnarray*}
H(\textbf{p}_{X/Y})
= 0.64150\ \ \text{nat}, \ \ 
R_2(\textbf{p}_{X/Y})= 0.28723\ \ \text{nat},\ \ \text{and}\ \ T_2(\textbf{p}_{X/Y})=0.24966 \ \ \text{nat}.
\end{eqnarray*}

\begin{table}
	\centering
		\begin{tabular}{c|c|c|c|c|c|c}
		\hline\hline
		\textcolor{white}{x} & & & & & &\\
		$(X,Y)$ & $(x_1,y_1)$& $(x_1,y_2)$ &$(x_2,y_1)$ & $(x_2,y_2)$& $(x_3,y_1)$& $(x_3,y_2)$\\
		$Z$ & $z_1$&$z_2$ &$z_3$ &$z_4$ &$z_5$ &$z_6$\\
		\textcolor{white}{x} & & & & & &\\
		$p_{Z,k}$& $\frac{3600}{5369}$& $\frac{3600}{900}$& $\frac{3600}{400}$& $\frac{225}{5369}$& $\frac{144}{5369}$& $\frac{100}{5369}$\\
		\textcolor{white}{x} & & & & & &\\
		\hline\hline
		\end{tabular}
		\vspace{0.4cm}
	\caption{Joint \textit{mpf} of $Z
	$}
	\label{tab2}
\end{table}

\noindent In our applications we simulated i.i.d. samples of size $n$ according to $\textbf{p}_Z$ and compute the CSE estimate $H(\widehat{\textbf{p}}_{(Y|X)}^{(n)})$, CRE estimate $R_\alpha(\widehat{\textbf{p}}_{(Y|X)}^{(n)})$, and CTE estimator $T_\alpha(\widehat{\textbf{p}}_{(Y|X)}^{(n)})$.

\noindent  \textsc{Figure} \ref{cseyx} concerns the CSE estimator $H(\widehat{\textbf{p}}_{(Y|X)}^{(n)})$ whereas \textsc{Figure} \ref{creyx} and \ref{cteyx} concern that of the CRE estimator $R_2(\widehat{\textbf{p}}_{(Y|X)}^{(n)})$ and $T_2(\widehat{\textbf{p}}_{(Y|X)}^{(n)})$.\\

\noindent In each of these \textsc{F}igures, left panels 
  represent plot of the proposed entropy estimator, built from sample sizes of $n=100,200,\cdots,30000$, and the true conditional entropy of $(Y|X)$  (represented by horizontal black line). We observe 
 that when  
 the sample sizes $n$ increase, then the proposed estimator value converges almost surely to the true value. \\
 \noindent  Middle panels show the histogram of the  data and where the red line represents the plots of the theoretical normal distribution calculated
 from the same mean and the same standard deviation of the data.\\
 \noindent Right panels concern the Q-Q plot of the data which display the observed values against normally 
distributed data (represented by the red line). We observe that the  underlying 
distribution of the data is normal since the points fall along a straight line.

 \section{Conclusion}
\label{sect_conclus} 
 In this paper, we described a method of estimating the joint probability mass function of a pair of discrete random variables. By the plug-in method, we constructed estimates of conditional Shannon-R\'eyni-Tsallis entropies and established almost sure rates of convergence and asymptotic normality of these ones. A simulation studies  confirm our results.
 \newpage

 \begin{center}
\begin{figure}[H]
\includegraphics[scale=0.28]{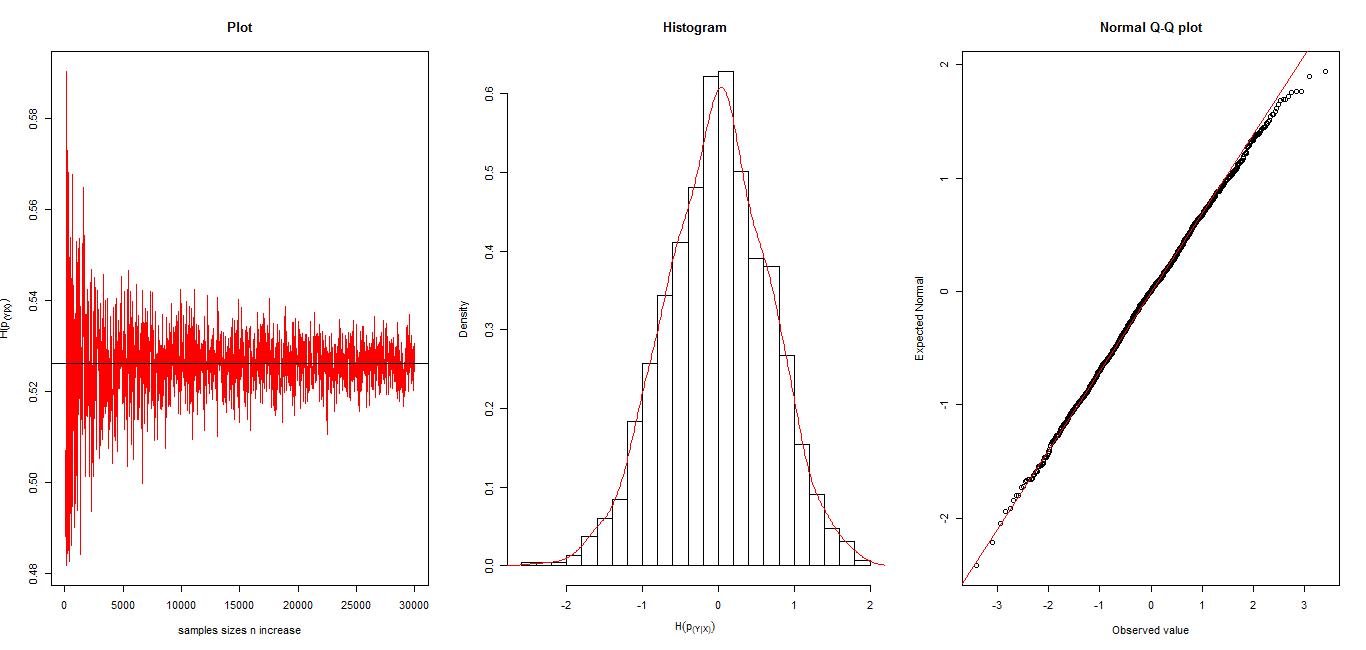} 
 \caption{Plot of $  H(\widehat{\textbf{p}}_{(Y|X)}^{(n)})$ when samples sizes increase, histogram and normal Q-Q plot  versus $\mathcal{N}(0,1)$.}\label{cseyx}
\end{figure}
\begin{figure}[H]
\includegraphics[scale=0.28]{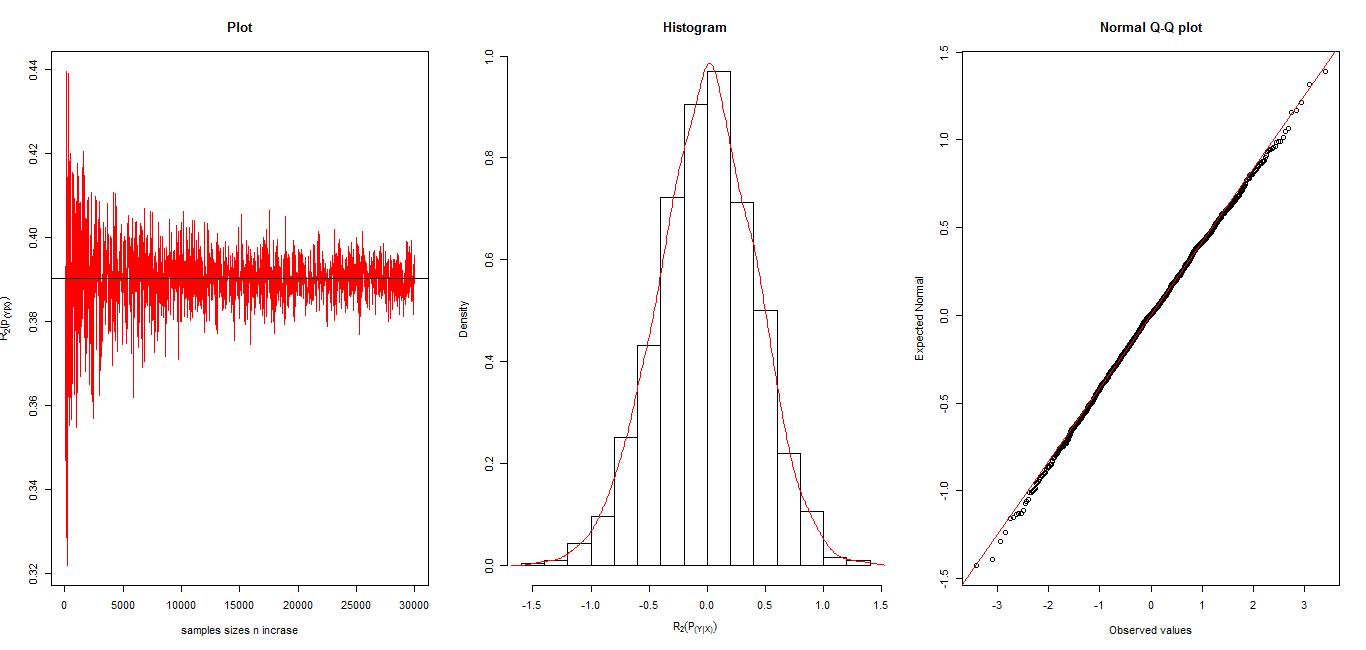} 
 \caption{Plot of $  R_2(\widehat{\textbf{p}}_{(Y|X)}^{(n)})$ when samples sizes increase, histogram and normal Q-Q plot  versus $\mathcal{N}(0,1)$.}\label{creyx}
\end{figure}
\begin{figure}[H]
\includegraphics[scale=0.28]{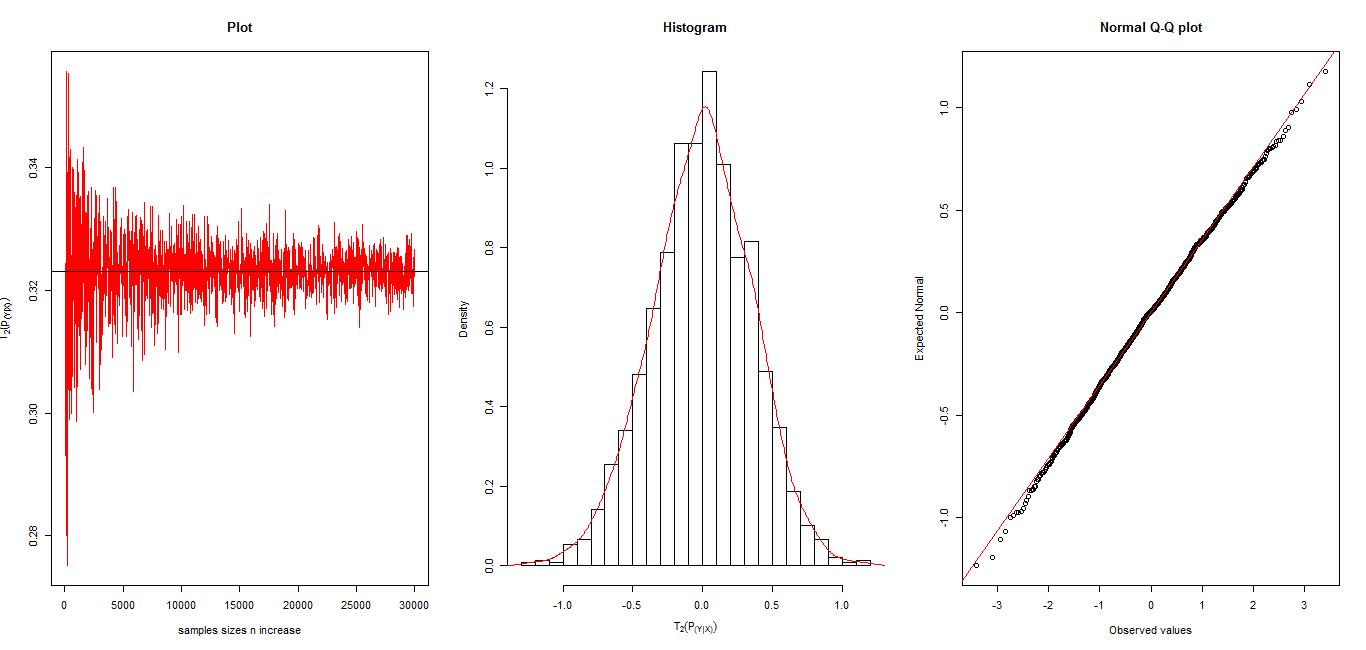} 
 \caption{Plot of $  T_2(\widehat{\textbf{p}}_{(Y|X)}^{(n)})$ when samples sizes increase, histogram and normal Q-Q plot  versus $\mathcal{N}(0,1)$.}\label{cteyx}
\end{figure}
\end{center}



\begin{thebibliography}{99}
\bibitem[Carter (2014)]{carter} T. Carter (2014). \textit{An introduction to information theory and entropy}. Complex Systems Summer School. Santa Fe.

 
\bibitem[Ba et Lo (2018)]{DiadiBa2018Divergence} Lo G.S and  Ba A.D (2018). Divergence measures estimation and its asymptotic normality theory
  using wavelets empirical processes $I$. \newblock {\em Journal of Statistical Theory and Applications}, \textbf{17}:158-171,
 
\bibitem[R\'enyi (1961)]{ren} A. R\'enyi (1961). On measures of entropy and information, in proc. \textit{4th Berkeley symp. mathematical statistics probability}. Berkeley, CA: univ. calif. press. \textbf{1}, 547-561.

\bibitem[Alshay \textit{et al.} (2014)]{Krishnamurthy2014Nonparametric}
A. Krishnamurthy, K. Kandasamy, B. Poczos, and L. Wasserman (2014). Nonparametric estimation of renyi divergence and friends. \textit{PMLR}, 919-927.

\bibitem[Cover and Thomas (1991)]{cove}
T. M. Cover, J. Thomas(1991). \textit{The elements of information theory}. John Wiley
and Sons, 1991.


\bibitem[Singh and Poczos (2014)]{Singh2014Generalized} S. Singh and B. Poczos (2014) Generalized exponential concentration inequality for renyi divergence
  estimation. \textit{PMLR},  333-341.

\bibitem[Hall (1987)]{Hall1987On} P. Hall (1987). On kullback-leibler loss and density estimation. \textit{The Annals of Statistics}, 15:1491-1519.

\bibitem[Tsallis (1988)]{tsal} C. Tsallis (1988). \textit{Possible generalization of Bolzmann-Gibbs statistics}, J. Stat.
Phys, 52, 479-487.


\bibitem[Tabass (2013)]{taba} M. S. Tabass, M. G. M. Borzadaran, M. Amini, (2013).
Conditional Tsallis Entropy. Cybernetics and Information Technologies,
Bulgarian Academy of Sciences, 13, 37-42.

\bibitem[Shannon (1948)]{shan} C. E. Shannon(1948). \textit{A mathematical theory of communication. Bell Syst}
.Techn. J. 27, 379-423, 623-656.


 \bibitem[Furuichi (2006)]{furu}
 S. Furuichi(2006), Information theoretical properties of Tsallis entropies, \textit{Journal of Mathematical Physics} 47, 023302.
 \bibitem[Abe (2000)]{abe} S. Abe (2000). Axioms and uniqueness theorem for Tsallis entropy, \textit{Phys. Lett. A}. 271 74-79.

\bibitem[Jayadev \textit{et al.} (2014)]{jaya} A. Jayadev, O. Alon, T.S. Ananda, and T. Himanshu(2014). The Complexity of Estimating R\'enyi Entropy. DOI: 10.1137/1.9781611973730.124 arXiv.

\bibitem[Manije (2013)]{mani}
S. Manije, M. Gholamreza and A. Mohammad (2013).  Conditional Tsallis Entropy, \textit{Cybernetics and Information Technologies} 13(\textbf{2}), pp:87-42. DOI: 10.2478/cait-2013-0012

\bibitem[Manije \textit{et al}. (2012)]{teix} 
A. Teixeira, A. Matos, L. Antunes (2012). Conditional Rényi entropies, \textit{IEEE Transactions on Information Theory} 58(\textbf{7}), 4273-4277.

\bibitem[Ba et Lo (2019)]{baentrop}
A.D. Ba  and G.S. Lo (2019). Entropies and their Asymptotic Theory in the discrete case. arXiv:1903.08645,submitted. 

\bibitem[Ba \textit{et al.} (2019)]{ba-mutual-entrop}
 A.D. Ba,  G.S. Lo, and C.T. Seck (2019). Joint, Renyi-Tsallis entropies and  mutual information estimation : Asymptotic limits. DOI: 10.13140/RG.2.2.28089.11368.

\bibitem[Csisz\'ar (1995)]{csis}
I.  Csisz\'ar, (1995). Generalized cutoff rates and R\'enyi's information measures, \textit{IEEE Transactions on Information Theory} 41(\textbf{1}), 26–34.


\bibitem[Philippatos $\&$ Wilson (1972)]{phil}
G.C.  Philippatos,  C.J. Wilson (1972). Entropy, market risk, and the selection of efficient portfolios. Appl. Econ., \textbf{ 4}, pp. 209–220.


\bibitem[Philippatos and Gressis (1975)]{phil2}
G. C. Philippatos and N. Gressis (1975). Conditions of formal equivalence among E-V, SSD, and E-H portfolio selection creteria: the case for uniform, normal, and lognormal distributions, \textit{Management Science} \textbf{11}, 617-625. 

\bibitem[Arackaparambil \textit{et al.} (2011)]{arack} 
C. Arackaparambil, S. B. Sergey, J. Joshua, A. Shubina (2011). Distributed Monitoring of Conditional Entropy
for Network Anomaly Detection. \textit{Dartmouth Computer Science Technical Report TR} 2009-653.

\bibitem[Cachin (1997)]{cachin}
 C. Cachin (1997). Entropy measures and unconditional security in cryptography. PhD thesis, Swiss Federal Institute of Technology Zurich.

\bibitem[Iwamoto and Shikata (2013)]{iwam}
M. Iwamoto, J. Shikata (2013) : Information Theoretic Security for Encryption Based on Conditional R\'enyi Entropies. IACR Cryptology ePrint Archive, 440 

\bibitem[Vollbrecht and Wolf(2002)]{vollb} 
K. G. H. Vollbrecht, M. M. Wolf (2002). Conditional entropies and their relation to entanglement criteria, \textit{Journal of Mathematical Physics} 43. 4299-4306.

\bibitem[Lake (2006)]{lake}
D.E. Lake  (2006). R\'enyi entropy measures of heart rate gaussianity, \textit{IEEE Transactions on Biomedical Engineering} 53  21-27.

\bibitem[Jizba and Arimitsu (2004)]{jizba}
P. Jizba, T. Arimitsu,(2004). The World According to Renyi Thermodynamics of Multifractal 
Systems. – \textit{Annals of Physics}, Vol. \textbf{312}, 17-59.

\bibitem[Golshani \textit{et al.} (2009)]{gols} L. Golshani, E. Pasha, G. Yari (2009). Some Properties of Renyi Entropy and Renyi Entropy Rate. – \textit{Information Sciences}, Vol. \textbf{179}, 2426-2433. 

\bibitem[Arimoto (1977)]{arim}
S. Arimoto(1977). Information Mesures and Capacity of Order $\alpha$ for Discrete Memoryless Channels. In Topics in
Information Theory; Colloquia Mathematica Societatis J\'anos Bolyai; Csisz\'ar, I., Elias, P., Eds.; J\'anos Bolyai
\textit{Mathematical Society and North-Holland}: Budapest, Hungary. \textbf{16}, pp. 493–519.

\bibitem[Renner and Wolf (2005)]{renn}
 R. Renner, S. Wolf. Advances in Cryptology-ASIACRYPT (2005). In Proceedings of the 11th International Conference on the Theory and Application of Cryptology and Information Security, Chennai, India, December 4–8, 2005; Chapter Simple and Tight Bounds for Information Reconciliation and Privacy Amplification; Springer: Berlin/Heidelberg, Germany, 199–216.

\bibitem[Bentes (2008)]{bent} 
S. R. Bentes, R. Menezes, D.A. Mendes(2008), Long memory and volatility clustering: is the empirical evidence consistent across stock markets?, \textit{Physica A} 387  3826-3830

\bibitem[Kanaya and Han (1995)]{kana}
F. Kanaya, T.S. Han (1995). The asymptotics of posterior entropy and error probability for Bayesian estimation, \textit{IEEE Transactions on Information Theory} 41  1988-1992.

\bibitem[Hayashi (2011)]{haya}
 M. Hayashi (2011). Exponential decreasing rate of leaked information in universal random privacy amplification. \textit{IEEE Trans. Inf. Theory} \textbf{57}, 3989–4001.

\bibitem[Cachin (1997)]{cach} C. Cachin,  Entropy Measures and Unconditional Security in Cryptography. Ph.D. Thesis, Swiss Federal
Institute of Technology Zurich, Zurich, Switzerland, 1997.

\bibitem[Lo (2016)]{ips-wcia-ang} 
Lo, G.S.(2016). Weak Convergence (IA). Sequences of random vectors.
\textit{SPAS Books Series}. Saint-Louis, Senegal - Calgary, Canada. Doi : 10.16929/sbs/2016.0001. Arxiv : 1610.05415. ISBN : 978-2-9559183- 1-9.
\end{thebibliography}
\end{document}